\documentclass[11pt]{amsart}
\usepackage[utf8]{inputenc}
\usepackage[english]{babel}
\usepackage{microtype}
\usepackage{tikz}
\usepackage{ytableau}
\usepackage{xcolor}
\usepackage{tikz-cd}
\usepackage{amsmath, amssymb, amsthm}
\usepackage{mathtools}
\usepackage{wrapfig}

\usepackage{cite}

\usepackage[onehalfspacing]{setspace} 

\usepackage[hyphens]{url}	
\RequirePackage[colorlinks,citecolor=blue,urlcolor=blue]{hyperref}
\hypersetup{breaklinks=true}
\usepackage[hyphenbreaks]{breakurl}

\theoremstyle{plain}
\newtheorem{theorem}{Theorem}[section]
\newtheorem{lemma}[theorem]{Lemma}
\newtheorem{proposition}[theorem]{Proposition}
\newtheorem{corollary}[theorem]{Corollary}
\newtheorem{definition}[theorem]{Definition}
\theoremstyle{remark}
\newtheorem{remark}[theorem]{Remark}
\newtheorem{example}[theorem]{Example}

\newcommand{\oeis}[1]{\href{http://oeis.org/#1}{#1}}

\newcommand{\dd}{\mathtt d}
\newcommand{\uu}{\mathtt u}
\newcommand{\hh}{\mathtt h}

\newcommand{\Bc}{\mathcal{B}}
\newcommand{\Dc}{\mathcal{D}}
\newcommand{\Pc}{\mathcal{P}}
\newcommand{\Sc}{\mathcal{S}}

\usepackage{tikz}
\usetikzlibrary{positioning}
\usetikzlibrary{arrows}
\tikzstyle{punkt}=[rectangle, rounded corners, draw=black, very thick, text centered]

\begin{document}
\baselineskip=17pt
\DeclarePairedDelimiter\floor{\lfloor}{\rfloor}

\title[Bijections and congruences involving lattice paths and integer compositions]{Bijections and congruences involving \\ lattice paths and integer compositions}
   \author[M.\ Ghosh Dastidar]{Manosij Ghosh Dastidar}
    
    \author[M.\ Wallner]{Michael Wallner}
    
    \email{manosij.dastidar@tuwien.ac.at, michael.wallner@tuwien.ac.at} 
 \address{TU Wien, Institute of Discrete Mathematics and Geometry, Wiedner Hauptstr.\ 8--10, 1040 Vienna, Austria}

\begin{abstract}
We prove new bijections between different variants of Dyck paths and integer compositions, which give combinatorial explanations of their simple counting formula $4^{n-1}$. 
These give relations between different statistics, such as the number of crossings of the $x$-axis in classes of Dyck bridges or the distribution of peaks in classes of Dyck paths, and furthermore relate them with $k$- and $g$-compositions.  
These allow us to find and prove congruence results for Dyck paths and parity results for compositions. Our investigation uncovers unexpected connections to mock theta functions, Hardinian arrays, little Schröder paths, Fibonacci numbers, and irreducible pairs of compositions, offering new insights into the structures of paths, partitions and compositions.
\vspace{0.3 cm}
\end{abstract}   

\maketitle


\section{Introduction}
This paper aims to establish a comprehensive framework connecting the fields of lattice paths, integer compositions, integer partitions and related combinatorial structures. Our investigation is motivated by the rich interplay between these distinct areas and the intriguing bijections and properties that link them. Therefore we start by defining the key concepts of the paper. 
A \emph{Dyck path} is a lattice path in $\mathbb{Z}^ 2$ from $(0,0)$ to $(2n,0) $ with up steps $\uu= (1,1)$ and down steps $\dd=(1,-1)$ that never crosses the $x$-axis. 
Dyck paths are ubiquitous in combinatorics and famously enumerated by the \textit{Catalan numbers}~$\frac{1}{n+1}\binom{2n}{n}$; see, e.g.,~\cite{Stanley2015Catalan}.
An \emph{integer partition} is a representations of a positive integer $n$ as a sum of positive integers, where different orders of the summands are not considered to be distinct. 
Similarly, an \emph{integer composition} is a representation of a positive integer $n$ as a sum of positive integers but in this case the different orders of the summands are considered to be distinct. 

Over the course of history many other structures have been defined to understand the world of paths, partitions and compositions. We provide a few of them here to aid the understanding of the subsequent sections of this paper:
\begin{enumerate}
    \item A \emph{Dyck bridge} is constructed when we relax the spatial restriction on Dyck paths such that it is allowed to traverse below the $x$-axis but must still end on the $x$-axis. 
    Historically, Dyck bridges are also sometimes called \emph{Grand Dyck paths}.
    \item A \emph{Dyck walk} is constructed when we further relax the condition of a bridge terminating on the $x$-axis, i.e., it is completely unconstrained and may end at any altitude. 
    \item A \emph{strict left-to-right maximum} is any peak ($\uu\dd$) that has a greater height than all peaks to its left. Similarly, a \emph{weak left-to-right maximum} is a peak that is greater than or equal to in height with all peaks to its left. 
    Note that all left-to-right maxima in this paper will be strict.
    \item A \emph{pair of compositions} of $n$ is an ordered pair $((a_1,\dots,a_k),(b_1,\dots,b_{\ell}))$ such that $n = a_1 + \dots + a_k =  b_1 + \dots + b_{\ell}$ are two compositions of $n$.
    \item A \emph{$k$-composition} of $n$, as introduced by Andrews~\cite{andrews2007theory}, is a composition of $n$ using parts made of $k$ distinct colors such that the last part is of the first color. 
\end{enumerate}

\begin{figure}[ht]
\begin{center}
\scalebox{0.52}{%
\newcommand{\myverticaldelta}{5mm}
    \begin{tikzpicture}[
        node distance = 3.0cm,
        auto,
        ->,
        >=latex',
        font=\LARGE,
        line width=0.8mm,
        block/.style = {circle, draw, minimum size=2.5cm, align=center},
    ]
    
        \node[block] (n1) {Pairs of \\compositions \\of $n$};
        \node[block, above left=\myverticaldelta and 2cm of n1] (n2) {$3$-compositions \\of $n$
        \cite{andrews2007theory}};
        \node[block, below left=\myverticaldelta and 2cm of n1] (n3) {$g$-compositions \\of $n$
        \cite{OuvryPolychronakos2019Exclusion, HopkinsOuvry2021Multicompositions}};
        \node[block, above right=\myverticaldelta and 2cm of n1] (n4) {Unconstrained \\Dyck walks \\ of length $2n-2$};
        \node[block, below right=\myverticaldelta and 2cm of n1] (n5) {$2$-colored \\Dyck bridges of \\ length $2n-2$ };
        \node[block, right=3cm of n4] (n6) {Dyck paths with \\height-labelled peaks \\of length $2n$};
        \node[block, right=3cm of n5] (n7) {Left-to-right maxima \\ in Dyck bridges \\ of length $2n$};
    
        \draw[<->] (n1) -- node[sloped, above, pos=0.5, font=\small]{\textit{Prop.~\ref{prop:threecompbijection}}} (n2);
   
        \draw[<->] (n1) -- node[sloped, above, pos=0.5, font=\small]{\textit{Prop.~\ref{prop:CompPairsDyckWalk}}} (n4);
        \draw[<->] (n2) -- node [sloped, above, pos=0.5]{\textit{Prop.~\ref{prop:KandG}}} (n3);
        \draw[<->] (n4) -- node[sloped, above, pos=0.5]{\textit{Prop.~\ref{prop:2colunconstrained}}}(n5);
        \draw[<->] (n5) -- node [sloped, above, pos=0.5]{\textit{Prop.~\ref{prop:markedlefttoright2coloredbridge}}} (n7);
        \draw[<->] (n6) -- node[sloped, above, pos=0.5, font=\small]{\textit{Prop.~\ref{prop:LabeltoMaxima}}} (n7);
            
    \end{tikzpicture}}
\end{center}
\caption{Bijections proved in this paper of classes of paths and compositions, all enumerated by $4^{n-1}$.}
\label{fig:bijections}
\end{figure}
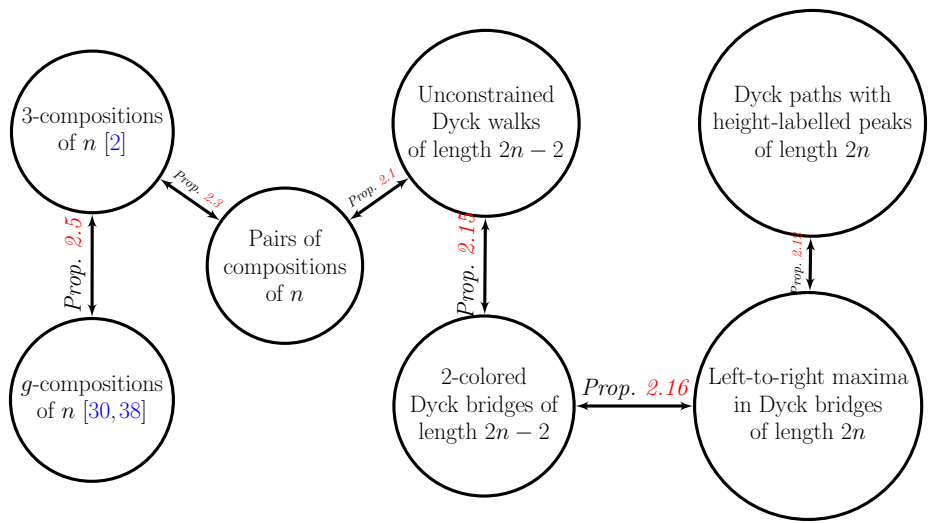

This paper is organized in three parts. 
In the first part, we prove a series of bijections visualized in Figure~\ref{fig:bijections}, linking the world of compositions and lattice paths by their enumeration formula $4^{n-1}$. 
We start
in Section~\ref{sec:bijectionscompositions}, where we give natural bijections between pairs of compositions, Andrews' $3$-compositions, $g$-compositions (a concept from quantum physics), and Dyck walks.  
We further explore bijections between Hardinian arrays and pairs of compositions with equal first parts. 
These bijections bring to light the commonalities between these usually disparate combinatorial objects. 
We continue in Section~\ref{sec:PeaksAndPaths} with the study of peaks and crossings of the $x$-axis in Dyck bridges. 
We present bijections between several classes of Dyck paths with specific peak properties, among which many are also in bijection with the compositions studied in Section~\ref{sec:bijectionscompositions}. 
Our results extend to the enumeration of Dyck bridges in restricted settings, highlighting the versatility of Dyck paths in combinatorial enumerations.
In Section~\ref{sec:compandDyckstrip}, we then derive various links between compositions and paths in a strip. 

In the second part, we focus on congruences in classes of Dyck paths.
Having established a series of bijections between walks and compositions raises the question of whether there are similar underlying structures at play here. For this reason, we ask the questions of whether famous results in the theory of partitions or compositions have analogous results in the theory of lattice paths. We are pleased to answer this question in the affirmative. 
In particular, in Section~\ref{sec:congruences}, we introduce the concept of a peak profile in Definition~\ref{def:peakprofile}, which is a statistic based on the number of peaks of a given height. 
This leads us to creating congruence relations for Dyck paths analogous to the celebrated congruence properties that have existed for the partition function. 
Our main result is Theorem~\ref{theo:rplus1div} based on a divisibility property of Dyck paths with exactly $r$ peaks per reached height.
This analysis also leads us to unexpected connections with little Schr\"oder paths and plane trees.

In the third part, we generalize the key concepts like peaks in Dyck paths and restrictions on parts in compositions.
In Section~\ref{sec:restrictedsummits}, we generalize peaks to summits, which may also occur at the beginning and the end of paths, and analyze them in Dyck bridges together with certain restrictions. 
We show how they are connected to other well-known objects like the Fibonacci numbers. 
In Section~\ref{sec:firstpart}, we consider limitations on the size of the first, greatest and smallest peaks or parts impact the overall nature of paths and partitions (or compositions). Euler's famous pentagonal number theorem motivated Fine to find a curiously analogous theorem for partitions into distinct parts where the first part has parity either odd or even. Extending this result to integer compositions we show in Theorem~\ref{theo:ConminusCenandTriangle} a surprising link back to closed walks. 
Inspired by these results we ask similar questions in the context of lattice paths and we discover an unexpected connection to integer compositions again. 

\section{Bijections between compositiosn and lattice paths}

\subsection{Bijections involving compositions and arrays}
\label{sec:bijectionscompositions}

There are instances in literature where connections between walks and compositions (including pairs of compositions) have been studied in detail. 
Bender, Lawler, Pemantle, and Wilf~\cite{bender2003irreducible} studied pairs of compositions in the context of probability of a first return to the origin of a random walk,
while B\'ona and Knopfmacher \cite{bona2010probability} have looked into the probability that pairs of compositions with part sizes $a$ and $b$ have the same length (i.e., number of parts) and in that context they have established bijections between these objects and weighted lattice paths. 
Dunkl \cite{dunkl2007hook} has given relations between pairs of compositions and hook-lengths in the modified Ferrer's diagram which could also be connected to walks via relations to adjusted Young's tableaux. 

The following straightforward bijection will serve as the link between lattice paths and compositions.

\pagebreak

\begin{proposition}
    \label{prop:CompPairsDyckWalk}
    There exists a natural bijection between pairs of compositions of $n$ and Dyck walks of length $2n-2$.
\end{proposition}

\begin{proof}
Let a pair $(A,B)$ of two compositions of $n$ be given.
First, we convert each composition to a binary sequence: 
For each element $k$ in a composition, append $k-1$ zeros followed by a~$1$.  
By construction, both of these sequences have to end in $1$. 
So we remove these ones and then concatenate the binary sequences, with $A$'s sequence coming first. 
Finally, after replacing each $0$ by an up step $\uu$ and each $1$ by a down step $\dd$ the claim follows. 
For the reverse direction cut the walk in the middle into two parts, and re-add the ones.
\end{proof}

Let us show how the above construction works on a concrete example. 

\begin{example}
Let the pair $(A,B)$ of compositions $A = (2, 1, 3)$ and $B = (3, 2, 1)$ of $6$ be given.
In the first step we get the two binary sequences
$011001$ and $001011$.
Then, after erasing the ones at the end we concatenate them and get $0110000101$.
\end{example}

Next, we show that pairs of compositions are in bijection with a special case of Andrews' $k$-compositions~\cite{andrews2007theory}, namely for $k=3$.

\begin{proposition}
    \label{prop:threecompbijection}
    There exists a natural bijection between $3$-compositions of $n$ and pairs of compositions of $n$.     
\end{proposition}

\begin{proof}
Anticipating the result, we use the following convention for the three colors $1,2,3$, anticipating a notion of left and right:
Remove the labels of color~$1$, 
use label $L$ for color~$2$,
and label $R$ for color~$3$.

Now we describe a map from $3$-compositions of $n$ to pairs of compositions of $n$.
First, we create two identical copies. 
In the first copy, we remove the labels $R$ and add the parts labeled by $L$ to the next part.
If the next part has also a label $L$, then the addition continues to the next part, etc. 
This gives a composition $A$ without any labels.
Similarly, in the second copy, we remove the labels $L$ and add the parts labeled by $R$ to the next part.
Again, if the next number has also a label $R$, then the addition continues, 
and we get a composition $B$ without any labels. 
Observe that the size of both compositions has not changed.
Therefore, $(A, B)$ is a pair of compositions of $n$. 

To prove that this map is in fact a  bijection, let us consider an arbitrary pair $(A,B)$ of compositions of $n$. 
The key statistic to consider is the \emph{run of identical parts}:
Let $A=(a_1,a_2,\dots,a_{\ell_A})$ and $B=(b_1,b_2,\dots,b_{\ell_B})$.
A run is sequence of maximal length such that $a_1=b_1$, $a_2=b_2$, and so on. 
If $a_1 \neq b_2$ we say the run has length $0$. 

For the inverse map, we will describe a recursive algorithm which reduces the sizes of $(A,B)$ step-by-step and builds a $3$-composition by distinguishing two cases.
First, assume $A=B$. 
Then we map the pair to $A$ with all parts having label $1$.     
Note, that in this case the run is trivial as it consists of the full composition.

Second, assume $A \neq B$.
Let the run have length $k$, i.e., $a_{k+1} \neq b_{k+1}$.
Now we keep the first $k$ identical parts and attach a label $1$.
If $a_{k+1}<b_{k+1}$ we keep $a_{k+1}$ and attach a label $3$;
if $a_{k+1}>b_{k+1}$ we keep $b_{k+1}$ and attach a label $2$.
Finally, we subtract the extracted sequence from the first parts of $A$ and $B$, and remove afterwards initial zeros.

We repeat this process with the new parts. 
As the sizes decrease in each step by at least one, this process terminates. 
Moreover, note that in each step both parts decrease by the same size. 
Hence, in the last step the process ends with case one where both parts are equal, and therefore the final part gets label $1$ as required in $3$-compositions.
\end{proof}

\medskip

\noindent\begin{minipage}{0.66\textwidth}
\begin{example}
Consider the $3$-composition $6_1+1_2+4_3+2_1$ of $n=13$.
First, we remove the labels of color~$1$, 
use label $L$ for color~$2$,
and label $R$ for color~$3$.
Second, we create two identical copies. 
In the first copy, we remove the labels $R$ and add the parts labeled by $L$ to the next part.
If the next part has also a label $L$, then the addition continues to the next part, etc. 
Similarly, in the second copy, we remove the labels $L$ and add the parts labeled by $R$ to the next part.
This gives a pair of compositions of $n$ without any labels, and we have shown in Proposition~\ref{prop:threecompbijection} that this is in fact a bijection. 
\end{example}
\end{minipage}
 \hfill
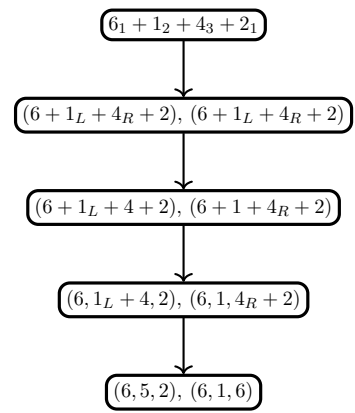
\begin{minipage}{0.3\textwidth}
\vspace{-0.2\baselineskip}
\begin{center}
\begin{tikzpicture}[node distance=0.75cm, auto, scale=0.5, every node/.style={scale=0.7}]
\node[punkt] (first) at (0,0.1) {\(6_1+1_2+4_3+2_1\)};
\node[punkt] (second) [below=of first] {\((6+1_L+4_R+2)\), \((6+1_L+4_R+2)\)};
\node[punkt] (third) [below=of second] {\((6+1_L+4+2)\), \((6+1+4_R+2)\)};
\node[punkt] (fourth) [below=of third] {\((6,1_L+4,2)\), \((6,1,4_R+2)\)};
\node[punkt] (fifth) [below=of fourth] {\((6,5,2)\), \((6,1,6)\)};

\path[thick,->] 
(first) edge (second)
(second) edge (third)
(third) edge (fourth)
(fourth) edge (fifth);
\end{tikzpicture}
\end{center}
\end{minipage}

\medskip

We end this discussion on compositions, by mentioning another class of compositions introduced by Ouvry and Polychronakos~\cite{OuvryPolychronakos2019Exclusion} while working on closed lattice random walks which confine a given area. 
A \emph{$g$-composition} is a composition that allows both positive integers and zeros, such that at most $g-2$ consecutive zeros may appear between positive parts. 
Hopkins and Ouvry show that $g$-compositions are in bijection with Andrews' $k$-compositions.
We adapt their proof slightly, by considering the zeros after a positive part.

\begin{proposition}[{\!\cite{HopkinsOuvry2021Multicompositions}}]
    \label{prop:KandG}
    There is a natural bijection between $k$-compositions and $g$-compositions with $g=k+1$.  
\end{proposition}

\begin{proof}
The number of zeros after a positive part $\lambda$ determine the color of $\lambda$.
If $\lambda$ is followed by $k$ zeros we map it to $\lambda_{k+1}$ (and vice versa). 
The order of appearance of positive parts remains unchanged. 
Since, in $g$-composition there are at most $g-2$ zeros between positive parts, these correspond in $k$-compositions to $k=g-1$ different colors.
\end{proof}

We have already seen how pairs of compositions serve as an important link between the fields of compositions and lattice paths. To reinforce this point we will add some more examples of how they are linked to predominant combinatorial and analytic objects. Generic pairs of compositions can be also be subdivided into two important pairs of subclasses: Concave compositions and convex compositions. These are ubiquitous in contemporary literature and are related to important concepts like mock theta functions and Hardinian arrays. 

\smallskip

{
\small 
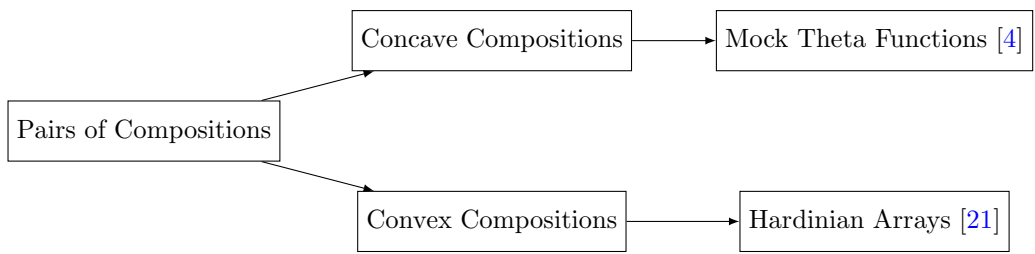
\begin{center}
\begin{tikzpicture}[scale=1.85, every node/.style={draw, rectangle, minimum width=3cm, minimum height=0.8cm}]
    \node (pairs) at (0,0) {Pairs of Compositions};
    \node (concave) at (2.5,0.65) {Concave Compositions};
    \node (convex) at (2.5,-0.65) {Convex Compositions};
    \node (mocktheta) at (5.25,0.65) {Mock Theta Functions~\cite{MR3048655}};
    \node (hardinian) at (5.25,-0.65) {Hardinian Arrays~\cite{dougherty2023hardinian}};

    \draw[->, >=latex] (pairs) -- (concave);
    \draw[->, >=latex] (pairs) -- (convex);
    \draw[->, >=latex] (concave) -- (mocktheta);
    \draw[->, >=latex] (convex) -- (hardinian);
\end{tikzpicture}
\end{center}
}

There have been historically two different definitions of concave compositions (which only slightly differ but the basic idea remains the same). 
Simply put, concave compositions are compositions where parts decrease up to the middle after which the parts increase again. Similarly convex compositions are those compositions where the parts increase up to the middle after which they weakly decrease till the end. 
Concave compositions are interesting because their generating function is a mixed mock modular form, which is related to the study of moonshine in Mathieu groups~\cite{zbMATH06254616}.
They are also related to partition functions where the greatest part is odd/even and also partition functions where the smallest non-appearing part is odd/even \cite{zbMATH05903054, MR3048655, MR2864448}. 

We will show now a new link between convex compositions and pairs of compositions to Hardinian arrays, which were introduced by Dougherty-Bliss and Kauers in~\cite{dougherty2023hardinian}.

\begin{definition}[Hardinian Arrays]
\label{def:Hardinian}
For any positive integer \( r \), let \( H_r(n, k) \) be the number of \( n \times k \) arrays which obey the following rules:
\begin{itemize}
    \item The entry in position \( (1, 1) \) is 0, and the entry in position \( (n, k) \) is \( \max(n, k) - r - 1 \).
    \item The entry in position \( (i, j) \) must equal or be one more than each of the entries in positions \( (i - 1, j) \), \( (i, j - 1) \), and \( (i - 1, j - 1) \).
    \item The entry in position \( (i, j) \) must be within \( r \) of \( \max(i, j) - 1 \).
\end{itemize}
\end{definition}

We consider the \emph{hooks} of the Hardinian array, i.e., the numbers appearing above an element of the diagonal and the numbers to the left of the diagonal. (These kind of shapes have been considered in the theory of partitions and compositions so we borrow the definition of a hook from that literature).   
The numbers in a hook in a Hardinian array are arranged such that they are weakly increasing up to the diagonal element whether looked from left to right or top to bottom. 
There are a number of operations that can be performed on a hook in a Young's diagram of a partition. 
The simplest one is just straightening the hook around the diagonal element which is used to give an elementary proof of why the number of partitions into odd, distinct parts is equal to the number of self-conjugate partitions\cite{zbMATH06775939}.
If we straighten the hooks in a Hardinian array we can immediately see that we get for each hook a convex composition; see Figure~\ref{fig:Hardin2ConvexComp}.

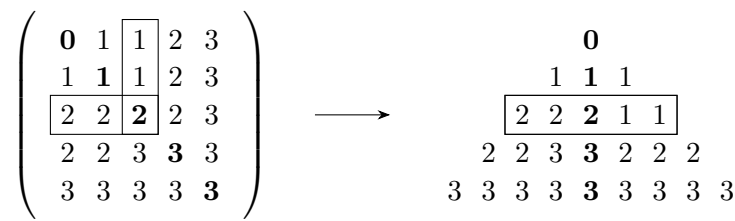
\begin{figure}[ht]
    \begin{center}
    \begin{tikzpicture}
    \matrix (m) [matrix of math nodes, left delimiter=(, right delimiter=)] {
    \boldsymbol{0} & 1 & 1 & 2 & 3\\
    1 & \boldsymbol{1} & 1 & 2 & 3\\
    2 & 2 & \boldsymbol{2} & 2 & 3\\
    2 & 2 & 3 & \boldsymbol{3} & 3\\
    3 & 3 & 3 & 3 & \boldsymbol{3}\\
    };
    \draw (m-3-1.north west) rectangle (m-3-3.south east);
    \draw (m-1-3.north west) rectangle (m-3-3.south east);
    
    \draw[-Stealth] ([xshift=1cm]m.east) -- ++(1,0);
    
    \matrix (mb) [matrix of math nodes] at (6,0) {
      &   &   &   & \boldsymbol{0} &   &   &   &   \\
      &   &   & 1 & \boldsymbol{1} & 1 &   &   &   \\
      &   & 2 & 2 & \boldsymbol{2} & 1 & 1 &   &   \\
      & 2 & 2 & 3 & \boldsymbol{3} & 2 & 2 & 2 &   \\
    3 & 3 & 3 & 3 & \boldsymbol{3} & 3 & 3 & 3 & 3 \\
    };
    \draw (mb-3-3.north west) rectangle (mb-3-7.south east);
    \draw (mb-3-3.north west) rectangle (mb-3-7.south east);
    
    \end{tikzpicture}
    \caption{An equivalent representation of a Hardinian array with $r=1$ as a triangle of convex compositions (obeying rules analogous to Definition~\ref{def:Hardinian}).}
    \label{fig:Hardin2ConvexComp}
\end{center}
\end{figure}

Dougherty-Bliss and Kauers show in~\cite[Theorem~1]{dougherty2023hardinian} that $H_1(n,n)= \frac{1}{3}(4^{n-1}-1)$ for $n\geq 1$, which is immediately connected to the $3$-compositions of Andrews we have just seen. 
If we restrict in any $3$-composition the first part to be the same color as the last part, then we automatically decrease the number of compositions by a factor of~$3$ (except the case where the first part and the last part are equal). 
Therefore the number of $3$-compositions with first and last part of the same color is equal to $\frac{1}{3}(4^{n-1}-1)$ = $H_1(n,n)$. 
Now, Proposition~\ref{prop:threecompbijection} motivates the search for pairs of compositions, or by Proposition~\ref{prop:CompPairsDyckWalk} for Dyck walks. 
Note that by the bijections, if in a $3$-composition the first and last color are the same, then the pairs of compositions have the same first part.

\begin{proposition}
    \label{prop:HardinBijection}
    There is a natural bijection between Hardinian arrays of size $n \times n$ with parameter $r=1$ and pairs of compositions of $n$ with at least two parts and  equal first part. 
\end{proposition}

\begin{proof}
    We are going to work in the equivalent number triangle introduced above.
    Then, the center column starts with a $0$, ends with a $n-2$, and increases weakly from top to bottom.
    Hence, one number is used twice and all the others only once. 
    Let $s\geq 2$ be the row of the second appearance. 
    By the rules of a Hardinian array in Definition~\ref{def:Hardinian}, the $i$th row contains numbers from $\{i-1,i\}$.
    Therefore, since each row is a convex composition, row $s$ and all rows below rows are deterministically filled with the minimum entry.

    It remains to consider the rows $2,3,\dots,s-1$, whose center column is $1,2,\dots,s-2$.
    Observe that the center column separates the triangle into a left and right one that are independent of each other.
    We will now map each side to a composition of $s-1$ as illustrated in Figure~\ref{fig:Hardin2CompPair}.
    For this purpose we define the following walk on the elements of the triangles.
    Start at the entry next to row $s-1$. 
    Recursively, for row $i$, if the current entry is $i$ go the the left/right, if it is $i-1$ go up.\footnote{Note that equivalent paths already appear in the original proof of \cite[Theorem~1]{dougherty2023hardinian}.}
    Count in each row the number of visited entries (including a final entry when leaving the triangle). 
    Read this sequence from bottom to top.
    As these walks have $s-1$ steps, this gives a composition of $s-1$, and each such composition corresponds to a unique walk. 
    This gives a composition $C_L$ for the left triangle, and $C_R$ for the right triangle. 
    Finally, attach the value $n-s+1$ as first part, and the claim follows. 

    For the reverse bijection directly follows from reversing these steps:
    The first part determines the center column, and the remaining parts the left and right triangle.
\end{proof}

\begin{figure}[ht]
\begin{center}
    \begin{tikzpicture}
    \matrix (mb) [matrix of math nodes] {
    \text{Row} & \text{Range} & \hspace{6mm} & \text{Comp.\ $C_L$} & \hspace{2mm} &    &   &   &   &   &   &   &   & &  \hspace{2mm} & \text{Comp.\ $C_R$}  \\
    1 & \{0\}    &&   && &&   &   & \boldsymbol{0} &   &   &   &  && \\
    2 & \{0,1\}  &&   && &&  & 1 & \boldsymbol{1} & 1 & \phantom{1}  &   &  && 2 \\
    3 & \{1,2\}  && 2 &  && \phantom{1} & 2 & 2 & \boldsymbol{2} & 1 & 1 &   &   && 1 \\
    4 & \{2,3\}  && 2 & && 2 & 2 & 3 & \boldsymbol{3} & 2 & 2 & 2 &   && 1 \\
    5 & \{3,4\}  && && 3 & 3 & 3 & 3 & \boldsymbol{3} & 3 & 3 & 3 & 3 && \\
    };
    \draw (mb-1-1.south west) -- (mb-1-16.south east);
    \draw (mb-6-6.north west) -- (mb-6-14.north east);
    
    \draw (mb-5-9) circle (5pt);
    \draw (mb-5-8) circle (5pt);
    \draw (mb-4-8) circle (5pt);
    \draw (mb-4-7) circle (5pt);
    
    \draw (mb-5-11) circle (5pt);
    \draw (mb-4-11) circle (5pt);
    \draw (mb-3-11) circle (5pt);
    \draw (mb-3-12) circle (5pt);
    \end{tikzpicture}
    \caption{The Hardinian array from Figure~\ref{fig:Hardin2ConvexComp} is in bijection with the pair of compositions $((1,2,2),(1,1,1,2))$ with equal first part $1$; see Proposition~\ref{prop:HardinBijection}.}
    \label{fig:Hardin2CompPair}
\end{center}
\end{figure}

\begin{remark}
Let $M(n)$ denote integer compositions with equal first and last part and at least two parts. Then we also find that $M(2n-1) = H_1(n,n)$. In general, $M(n)$ is equal to the number of tilings of a $2\times n$ rectangle with $1\times 2$ and $2 \times 1$ dominoes and $2\times 2$ squares, which gives the Jacobsthal numbers \oeis{A001045} and interestingly is also equal to Hopkins and Ouvry's $g$-compositions ($g=2$) of $n+1$ without any 1s.  
\end{remark}

Dougherty-Bliss and Kauers~\cite{dougherty2023hardinian} also prove that $H_r(n,n)$ satisfies linear recurrences with polynomial coefficients for $r \geq 2$, which, however, do not admit simple closed forms. 
But they show that for $r=1$ the rectangular case has the solution $H_1(n, k) = 4^{k-1} (n - k) + H_1(k,k)$ for $n \geq k \geq 1$, which we also suspect to have an interpretation as a family of compositions.

\smallskip

In the next section we will extend the connection from Proposition~\ref{prop:CompPairsDyckWalk} between pairs of compositions and Dyck paths by proving further bijections between different classes of lattice paths, among which some are enumerated by $4^{n-1}$.
In these classes, the distribution of peaks will play a key role.

\subsection{Bijections involving peaks and crossings in paths}
\label{sec:PeaksAndPaths}

Let us first consider the number of peaks $\uu\dd$ in all Dyck paths. 
A peak has \emph{height $h$} if its $\uu$ step ends on height $h$.
A \emph{Dyck path with marked peak} is a Dyck path in which one peak gets a special marker. We consider such paths to be different, if they are different paths, or, if they consist of the same path but have different marked peaks.

Moreover, we will need the concept of \emph{negative Dyck paths}: these consist of steps $\uu$ and $\dd$, start at the origin, end on the $x$-axis, and always stay weakly below the $x$-axis. 
Note that there are many known bijections between Dyck paths and negative Dyck paths, like, e.g., 
flipping the path by swapping each $\dd$ by $\uu$ and vice versa; 
or, alternatively, rotating the full path by $180^{\circ}$.

In our first result we connect Dyck paths with marked peaks and Dyck bridges (i.e., paths ending at $0$ that may traverse below the $x$-axis).

\begin{theorem}
    \label{theo:Dyckpeak}
    There is an explicit bijection between Dyck paths with marked peak of height~$h$ and Dyck bridges starting with a $\dd$ step and $h-1$ crossings of the $x$-axis preserving the length.
    Therefore, the number of peaks in all Dyck paths of length $2n$ is equal to $\binom{2n-1}{n}$.
\end{theorem}

    \newcommand{\DD}{D}
    \newcommand{\LL}{L}
    \newcommand{\RR}{R}
\begin{proof}
    First, let a Dyck path $\DD$ with marked peak at height $h$ be given. 
    Using this peak, we decompose the path $\DD$ into a left part $\LL$ from the origin to this peak and a right part $R$ from this peak to the end: $\DD = \LL \RR$ such that $\LL$ ends with $\uu$ and $\RR$ starts with $\dd$.
    In $\LL$ we perform a last-passage decomposition, cutting at the $\uu$ leaving a certain altitude for the last time;
    while in $\RR$ we perform a first-passage decomposition, cutting at the $\dd$ bringing us down to a new altitude the first time; see Figure~\ref{fig:firstlastpassageLR}.
    More formally, we have
    \begin{align}
        \label{eq:DyckPeakDecompLR}
        \begin{aligned}
        \LL &= \LL_1 \uu \LL_2 \uu \dots \uu \LL_{h} \uu,\\
        \RR &= \dd \RR_{h} \dd \RR_{h-1} \dd \dots \dd \RR_{1},
        \end{aligned}
    \end{align}
    where $\LL_i$ and $\RR_i$ for $j=1,\dots,h$ are Dyck paths.
    Now, we pair the paths $\LL_i\uu$ and $\dd \RR_i$ with the same index and map them to non-empty Dyck paths 
    \begin{align*}
        \DD_i = \uu \LL_i \dd \RR_i.
    \end{align*}    
    Then we concatenate these parts, after mapping each other part to its negative pair using any fixed bijection $\varphi$ between Dyck paths and negative Dyck paths. 
    This gives the Dyck bridge
    \begin{align}
        \label{eq:DyckPeakFinalDecomp}
        \varphi(\DD_1) \DD_2 \varphi(\DD_3) \DD_4 \dots \varphi(\DD_{h-1}) \DD_{h}
    \end{align}
    when $h$ is even.
    For odd $h$ it ends with $\varphi(\DD_{h})$.
    This bridge starts with a down  step $\dd$  and crosses the $x$-axis $h-1$ times, as claimed.

   \begin{figure}[t]
    \begin{tikzpicture}[scale=0.25]
        
        \newcommand{\pathLength}{49}
        \draw[step=1cm,lightgray,very thin] (0,0) grid (\pathLength,8);
        \draw[thick,->] (0,0) -- (\pathLength+0.5,0) node[anchor=north west] {x};
        \draw[thick,->] (0,0) -- (0,8+0.5) node[anchor=south east] {y};

        \draw[line width=2pt,blue] (0,0) -- ++(1,1); 
        \draw[line width=2pt,blue] (1,1) -- ++(1,1) -- ++(1,-1) -- ++(1,1); 
        \draw[line width=2pt,blue] (4,2) -- ++(1,1) -- ++(1,-1); 
        \draw[line width=2pt,blue] (6,2) -- ++(1,-1); 
        \draw[line width=2pt,blue] (7,1) -- ++(1,-1) -- ++(1,1); 
        \draw[line width=2pt,black] (8,0) -- ++(1,1) -- ++(1,1); 
        \draw[line width=2pt,blue] (10,2) -- ++(1,1); 
        \draw[line width=2pt,blue] (11,3) -- ++(1,1) -- ++(1,-1) -- ++(1,1); 
        \draw[line width=2pt,blue] (14,4) -- ++(1,1); 
        \draw[line width=2pt,blue] (15,5) -- ++(1,1); 
        \draw[line width=2pt,blue] (16,6) -- ++(1,1); 
        \draw[line width=2pt,blue] (17,7) -- ++(1,1) -- ++(1,-1); 
        \draw[line width=2pt,blue] (19,7) -- ++(1,-1); 
        \draw[line width=2pt,blue] (20,6) -- ++(1,-1); 
        \draw[line width=2pt,blue] (21,5) -- ++(1,-1); 
        \draw[line width=2pt,blue] (22,4) -- ++(1,-1); 
        \draw[line width=2pt,blue] (23,3) -- ++(1,-1) -- ++(1,1);
        \draw[line width=2pt,black] (24,2) -- ++(1,1) -- ++(1,1); 
        \draw[line width=2pt,black] (26,4) -- ++(1,1); 
        \draw[line width=2pt,blue] (27,5) -- ++(1,1) -- ++(1,-1) -- ++(1,1); 
        \draw[line width=2pt,black] (29,5) -- ++(1,1) -- ++(1,-1) -- ++(1,-1); 
        \draw[line width=2pt,blue] (32,4) -- ++(1,-1) -- ++(1,1)-- ++(1,1); 
        \draw[line width=2pt,black] (32,4) -- ++(1,-1); 
        \draw[line width=2pt,blue] (35,5) -- ++(1,1); 
        \draw[line width=2pt,blue] (36,6) -- ++(1,1); 
        \draw[line width=2pt,blue] (37,7) -- ++(1,1) -- ++(1,-1); 
        \draw[line width=2pt,blue] (39,7) -- ++(1,-1);
         \draw[line width=2pt,blue] (40,6) -- ++(1,-1); 
        \draw[line width=2pt,blue] (41,5) -- ++(1,-1); 
        \draw[line width=2pt,blue] (42,4) -- ++(1,-1) -- ++(1,-1) -- ++(1,1) -- ++(1,-1); 
        \draw[line width=2pt,black] (43,3) -- ++(1,-1); 
         \draw[line width=2pt,black] (46,2) -- ++(1,-1); 
         \draw[line width=2pt,black] (47,1) -- ++(1,-1); 
        
        \filldraw[red] (30,6) circle (8pt);
        \draw[red, dashed] (30,0) -- (30,6);

        \node at (8,6) {L};
        \node at (42,6) {R};
    \end{tikzpicture}

      \begin{tikzpicture}[scale=0.25]
    \newcommand{\pathLength}{49}
        \draw[step=1cm,lightgray,very thin] (0,-7) grid (\pathLength,5);
        \draw[thick,->] (0,0) -- (\pathLength+0.5,0) node[anchor=north west] {x};
        \draw[thick,->] (0,-7) -- (0,5+0.5) node[anchor=south east] {y};
    
    \draw[line width=2pt,blue] (0,0)
    -- ++(1,-1)
    -- ++(1,-1)
    -- ++(1,-1)
    -- ++(1,1)
    -- ++(1,-1)
    -- ++(1,-1)
    -- ++(1,1)
    -- ++(1,1)
    -- ++(1,1)
    -- ++(1,1)
    -- ++(1,1)
    -- ++(1,-1)
    -- ++(1,-1)
    -- ++(1,-1)
    -- ++(1,-1)
    -- ++(1,1)
    -- ++(1,-1)
    -- ++(1,-1)
    -- ++(1,-1)
    -- ++(1,-1)
    -- ++(1,-1)
    -- ++(1,1)
    -- ++(1,1)
    -- ++(1,1)
    -- ++(1,1)
    -- ++(1,1)
    -- ++(1,1)
    -- ++(1,1)
    -- ++(1,-1)
    -- ++(1,1)
    -- ++(1,1)
    -- ++(1,-1)
    -- ++(1,1)
    -- ++(1,1)
    -- ++(1,1)
    -- ++(1,1)
    -- ++(1,1)
    -- ++(1,-1)
    -- ++(1,-1)
    -- ++(1,-1)
    -- ++(1,-1)
    -- ++(1,-1)
    -- ++(1,-1)
    -- ++(1,1)
    -- ++(1,1)
    -- ++(1,1)
    -- ++(1,-1)
    -- ++(1,-1)
;

    \draw[line width=2pt,black] (0,0) -- (1,-1); 
    \draw[line width=2pt,black] (9,-1) -- ++(1,1) -- ++(1,1) -- ++(1,-1) -- ++(1,-1); 
    \draw[line width=2pt,black] (27,-1) -- ++(1,1); 
    \draw[line width=2pt,black] (30,0) -- ++(1,1) -- ++(1,-1); 
    \draw[line width=2pt,black] (42,0) -- ++(1,-1) -- ++(1,1) -- ++(1,1); 
    \draw[line width=2pt,black] (47,1) -- ++(1,-1);

    \filldraw[red] (10,0) circle (8pt);
    \filldraw[red] (12,0) circle (8pt);
    \filldraw[red] (30,0) circle (8pt);
    \filldraw[red] (42,0) circle (8pt);
    \filldraw[red] (44,0) circle (8pt);
    
\end{tikzpicture}
    \caption{Dyck path with a marked peak (red dot) at height 6 and image under bijection from Theorem~\ref{theo:Dyckpeak} given by a Dyck bridge starting with a $\dd$~step and 5=6-1 crossings (red dots). The black steps are used in the last-passage (resp., first-passage) decomposition in the proof. } 
    \label{fig:firstlastpassageLR}
\end{figure}
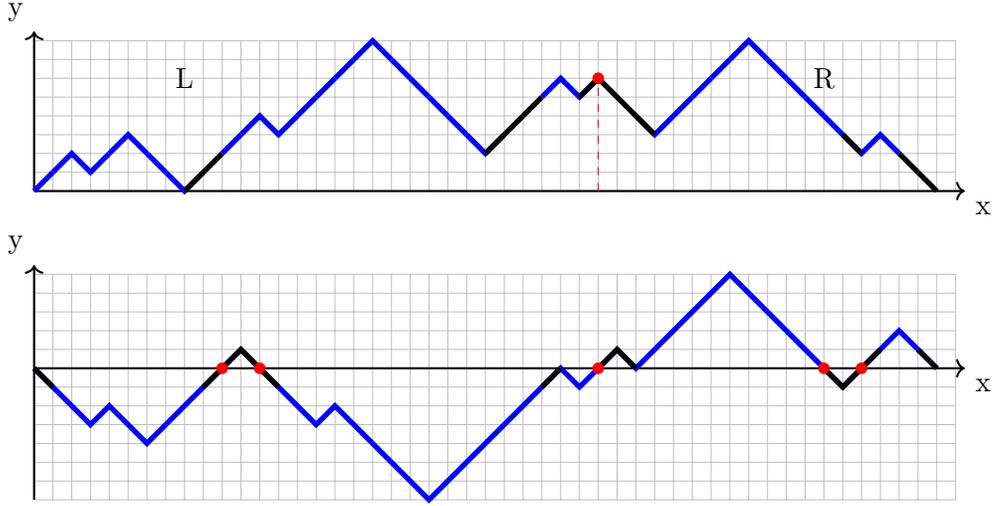

    Second, let a Dyck bridge starting with a $\dd$ step be given. 
    We cut at each crossing of the $x$-axis and recover the components $\DD_i$ and $\varphi(\DD_i)$. 
    Hence, it is straightforward to recover the components $\LL_i$ and $\RR_i$ and to rebuild the Dyck path $\DD$ with marked peak. 

    Finally, bridges of length $2n$ are counted by $\binom{2n}{n}$, as there is an equal number of up and down steps. As half of them start with a down step, we get $\frac{1}{2}\binom{2n}{n} = \binom{2n-1}{n}$.
\end{proof}

\begin{remark}
    The choice of $\varphi$ and the mapped pairs $L_i$ and $R_i$ that form $D_i$ is arbitrary and many other choices are possible, which would lead to different bijections between these two classes.
    For the map $\varphi$ any bijection between Dyck paths and negative Dyck paths is usable.
    Moreover, the above proof immediately also gives a bijection to Dyck bridges starting with an $\uu$ step, by applying $\varphi$ in \eqref{eq:DyckPeakFinalDecomp} only to components with even index. 
\end{remark}

\begin{remark}
    The counting sequence for the total number of peaks in Dyck paths is \oeis{A001700} in the OEIS\footnote{The On-Line Encyclopedia of Integer Sequences: \url{http://oeis.org/}}, and counts many other combinatorial objects.
    One of those is a special class of pairs of compositions, namely
    pairs $(A,B)$ of compositions of $n$ such that $A$ and $B$ have equal number of parts or $A$ has one more part than $B$.
    It is easy to see that these pairs are in bijection with compositions of $2n$ such that the sum of the elements at odd positions is equal to the sum of the elements at even positions, mentioned in~\oeis{A001700}.
\end{remark}

Next we consider the sum of all heights in all Dyck paths of length $2n$.
Our chain of bijections will explain its remarkably simple formula $4^{n-1}$.
For this purpose we generalize the class of Dyck paths with marked peak. 
A \emph{Dyck path with height-labeled peak}, is a Dyck path in which one peak gets a label from $\{1,2,\dots,h\}$ where $h$ is the height of the specific peak.

\begin{proposition}
\label{prop:LabeltoMaxima}
    There is an explicit bijection between Dyck paths with height-labeled peak with label $\mu$ at height $h$ and Dyck bridges of the same length with marked strict left-to-right maximum at height $\mu$ and $h-\mu$ crossings of the $x$-axis after this maximum.
\end{proposition}

\begin{proof}
    This bijection follows directly from the one described in the proof of Theorem~\ref{theo:Dyckpeak}, whose notation we will use here.
    The difference is that here we concatenate the (positive and negative) Dyck paths differently.

    Let a Dyck path with height-labeled peak be given.
    Let $h$ be the height of this peak and $\mu \in \{1,\dots,h\}$ be its label.
    First, we apply the bijection $\varphi$ onto all parts $\LL_i$ in \eqref{eq:DyckPeakDecompLR}.
    From that we get the following bridge in which the height-labeled peak is now a left-to-right maximum (underlined):
    \begin{align*}
        \varphi(\LL_1) \uu \varphi(\LL_2) \uu \dots \uu \varphi(\LL_{h})
        \, \underline{\uu \dd} \,
        \RR_{h} \dd \RR_{h-1} \dd \dots \dd \RR_{1}.
    \end{align*}
    Next, we transform this bridge, such that in the end the height-label $\mu$ constitutes the height of the left-to-right maximum.
    For this purpose, we create and concatenate the paths $\DD_i$ and $\varphi(\DD_i)$ in an alternating fashion at the end:
    \begin{align}
        \label{eq:heightlabeled2LRbridge}
        \varphi(\LL_1) \uu \varphi(\LL_2) \uu \dots \uu \varphi(\LL_{\mu})
        \, \underline{\uu \dd} \,
        \RR_{\mu} \dd \RR_{\mu-1} \dd \dots \dd \RR_{1} 
        \, \varphi(\DD_{\mu+1}) \DD_{\mu+2} \varphi(\DD_{\mu+3}) \dots \DD_h,
    \end{align}
    when $h-\mu$ is even. Otherwise, the last $\DD_h$ is replaced by $\varphi(\DD_h)$.

    For the reverse direction, let a Dyck bridge with marked left-to-right maximum be given. 
    It is then straightforward to decompose it into~\eqref{eq:heightlabeled2LRbridge} and to reverse the steps above to build a Dyck path.
    The left-to-right maximum becomes the height-labeled peak, labeled by its current height. 
    Observe that the height-labeled peak is lifted by the number of crossings of the $x$-axis to the left of this peak.
\end{proof}

Next, we connect our results with yet another class of paths: \emph{$2$-colored bridges}; see~\cite[Section 6.4]{BanderierKubaWallner2022}.
They are defined as the concatenation of two bridges such that the first bridge is colored in color $1$ and the second one in color $2$.
Note that contrary to \cite{BanderierKubaWallner2022}, we allow each part to be empty.
Hence, it is easy to see that its generating function is equal to the square of the  generating function $B(z) = \frac{1}{\sqrt{1-4z^2}}$ of bridges:
\[
    B(z)^2 = \frac{1}{1-4z^2}.
\]
The following Proposition~\ref{prop:2colunconstrained} provides a combinatorial explanation for this equality with the generating function of unconstrained walks.
For its proof we will need the following lemma, which we enrich by the statistics of \emph{negative returns to zero}, which are $\uu$ steps that end on the $x$-axis, i.e., they come from below.

\begin{lemma}[{\!\!\cite{Marchal2003}}]
    \label{lem:bijbridgesmeanders}
    There is an explicit bijection between Dyck bridges with $k$ negative returns to zero and Dyck meanders of the same length~$2n$ ending at altitude~$2k$.
    Therefore the number of such paths is $\binom{2n}{n}$.
\end{lemma}

\begin{proof}
    See \cite[p.\ 185 and Fig.~3]{Marchal2003}.
    Note that in this paper negative Dyck paths are defined as minimal Dyck paths, i.e., they touch the $x$-axis only at the beginning and at the end: $\dd \varphi(\DD) \uu$. 
    However, classical Dyck paths (i.e., stay always weakly above the $x$-axis) do not need to be minimal and are therefore either empty or the shape $\uu \DD \dd \DD$.
\end{proof}

\begin{remark}[Alternative bijection between even-length Dyck bridges and meanders]
    A different bijection can be derived using the ideas developed for Theorem~\ref{theo:Dyckpeak}. It then connects the number of crossings in the Dyck bridge with the final altitude of the Dyck meander as follows.
    Given a Dyck bridge, cut at each crossing of the $x$-axis. 
    This gives decomposition~\eqref{eq:DyckPeakFinalDecomp} with an additional Dyck path $\DD_0$ at the beginning, that could be empty in contrast to the $\DD_i$ with $i > 0$. 
    Then, we get a meander ending at altitude $2h$ by flipping each negative Dyck path $\varphi(D_i)$ and replacing its final step by $\uu$.
    If the initial excursion $\DD_0$ is non-empty, then $h$ was the number of crossings of the $x$-axis, while, 
    if $\DD_0$ is empty, then the number of crossings was $h-1$.    
\end{remark}

\begin{proposition}
    \label{prop:2colunconstrained}
    There is an explicit bijection between $2$-colored Dyck bridges and unconstrained Dyck walks of the same length $2n$.
\end{proposition}

\begin{proof}
    In all cases we will use Lemma~\ref{lem:bijbridgesmeanders} to transform bridges into meanders.
    Sometimes, it will be necessary to transform a meander further into a negative meander, by flipping all steps, i.e., exchanging $\uu$ by $\dd$ and vice versa. 
    
    We distinguish four cases.    
    First, the first and second bridges are non-empty. 
    The idea is that the change in color corresponds to the last crossing of the $x$-axis.    
    For this purpose we transform the second bridge into a meander or negative meander and 
    attach it to the first bridge such that the attached meander continues on the other side of the $x$-axis. 
    We can easily reverse this procedure by cutting at the last crossing of the $x$-axis.
    All the other cases will have no crossings of the $x$-axis.
    Second, if the first bridge is non-empty and the second one is empty, we transform the first bridge into a meander.
    Third, if the first bridge is empty and the second on is non-empty, we transform the second bridge into a negative meander.
    Finally, if both bridges are empty, we map them to the empty walk.    
\end{proof}

The following proposition connects Dyck bridges with marked strict left-to-right maximum and $2$-colored Dyck bridges and ends the chain of bijections of objects all enumerated by $4^{n-1}$.
For a proof using generating functions see~\cite[Theorem~1]{BlecherKnopfmacherProdinger2024Lefttoright}.

\begin{proposition}
    \label{prop:markedlefttoright2coloredbridge}
    There is an explicit bijection between Dyck bridges of length $2n$ with marked strict left-to-right maximum at height $h$ and $2$-colored Dyck bridges of length $2n-2$ with $h-1$ crossings of the $x$-axis in color $1$.
\end{proposition}

\begin{proof}
    Let us start with a Dyck bridge with marked strict left-to-right maximum of length $2n$.
    Then, we cut the bridge at the first return to the $x$-axis after this maximum. 
    The second part to the right is a bridge, which we give color $2$.
    Onto the first part we apply a similar idea as in the bijection of Theorem~\ref{theo:Dyckpeak}.
    As before, we cut the path at the marked left-to-right maximum into a left and right part given by $\LL \RR$, such that $\LL$ ends with $\uu$ and $\RR$ starts with~$\dd$.
    Then, we further decompose it similar to~\eqref{eq:DyckPeakDecompLR} into
    \begin{align*}
        \LL &= \LL_1 \uu \LL_2 \uu \dots \uu \LL_{h} \uu,\\
        \RR &= \dd \RR_{h} \dd \RR_{h-1} \dd \dots \RR_{2} \dd,
    \end{align*}
    where $h$ is the height of the peak, the $\LL_i$ are negative excursions, and the $\RR_i$ are excursions.
    Note that in this case $\RR$ ends with a $\dd$ step and contains only $h-1$ Dyck paths $\RR_i$.
    As in the proof of Theorem~\ref{theo:Dyckpeak} we form Dyck paths $D_i = \uu \varphi(\LL_i) \dd \RR_i$ for $i=2,\dots,h$, however, we need to flip the parts $\LL_i$ here.
    Finally, we remove the two steps of the marked peaks, and get the following bridges with $2$ steps less (compare with Equation~\eqref{eq:DyckPeakFinalDecomp}):    
    \begin{align*}
        \LL_1 \DD_2 \varphi(\DD_3) \DD_4 \dots \varphi(\DD_{h-1}) \DD_h,
    \end{align*}
    when $h$ is even.
    For odd $h$ it ends with $\varphi(\DD_h)$.
\end{proof}

After having proven the chain of bijections shown in Figure~\ref{fig:bijections}, we will explore connections between Dyck paths with more space constraints and integer compositions in the next section.

\subsection{Compositions and paths in a strip}
\label{sec:compandDyckstrip}

In this section, we will explore connection between integer compositions and Dyck paths and bridges in strips.

\begin{theorem}
    \label{theo:Dyckbound2}
    There is an explicit bijection between Dyck paths of length $2n$ and height at most~$2$ with $k$ returns to zero and integer compositions of~$n$ with $k$ parts.
\end{theorem}

\begin{proof}
    Let a Dyck path $\Dc$ with the claimed properties be given.
    We cut at each return to zero and decompose it into a sequences of arches: $\Dc = A_1 A_2 \dots A_k$. 
    Let now $p_{2,i}\geq 0$ be the number of peaks at level $2$ of $A_i$.
    Note that each $\uu$ step that does not leave the $x$-axis, contributes to a peak by the constraint of height at most $2$.
    Hence the \emph{$2$-peak profile} $(p_{2,i})_{i=1}^{k}$ uniquely characterizes $\Dc$.
    Finally, we need to shift the $2$-peak profile by $+1$ to get the claimed integer composition of $n$. Note that this $+1$ may be associated to the unique $\uu$ leaving the $x$-axis. Hence, each $\uu$ step contributes to exactly one part.
    \begin{align*}
        n = \left( p_{2,1}+1 \right) + 
            \left( p_{2,2}+1 \right) + 
            \dots +
            \left( p_{2,k}+1 \right).
    \end{align*}
    For the inverse bijection one simply reverses the above steps.
\end{proof}

The previous bijection yields the following interpretation of Andrews' $k$-compositions.
\begin{corollary}
    \label{coro:Dyckbound2}
    There is an explicit bijection between Dyck paths of length $2n$ and height at most $2$ in which each return to the $x$-axis bar the last one has one of $k$ colors and Andrews' $k$-compositions. 
    Therefore, they are enumerated by $(k+1)^{n-1}$.
\end{corollary}

\begin{proof}
    By Theorem~\ref{theo:Dyckbound2} each return corresponds to one part, hence, each color may be associated to this part and the claim holds.
    An alternative proof of the counting formula follows directly from the explicit generating function of the considered Dyck paths.
    %
\end{proof}

We can further generalize this class of Dyck paths as follows.

\begin{theorem}
    The number of Dyck bridges of length $2n$ in the strip $[-2,2]$ such that each crossing of the $x$-axis is colored by one of $k$ colors, is equal to $2(k+2)^{n-1}$
\end{theorem}

\begin{proof}
    Note that half of the paths start with an up step and half of them with a down step, which explains the factor $2$.
    Let us enumerate the paths starting with an up step.
    First, we cut at each crossing of the $x$-axis.
    Each part (after a flip along the $x$-axis) is an instance of a Dyck path of length $2n_i$ of height at most $2$.
    Hence, by Theorem~\ref{theo:Dyckbound2} the $i$th part is in bijection with a composition of $n_i$ and therefore enumerated by $2^{n_i-1}$.
    Finally, we attribute the color of the crossing to the part starting at that crossing. 
    Hence, each, except the first part, may have one of $k$ colors.
    Therefore, the total number of such paths is
    \begin{align*}
        \sum_{i=1}^n \sum_{\substack{n_1 + \dots + n_i = n \\ n_i \geq 1}} k^{i-1} 2^{n_1 + \dots + n_i-i}
        = 2^{n-1} \sum_{i=1}^n \left( \frac{k}{2} \right)^{i-1} \binom{n-1}{i-1}
        = \left( k + 2 \right)^{n-1},
    \end{align*}
    proving the claim.
\end{proof}

We leave it as an open problem to build an explicit bijection between these bridges starting with an up step and $(k-1)$-compositions.
Inspired by our proved bijections, we will study arithmetic properties of classes of Dyck paths that are hard to enumerate in the next section.

\section{Congruences and Peak Profiles for Paths}
\label{sec:congruences}

Arithmetic properties of integer compositions and partitions have been studied in detail for the greater part of the last century \cite{MR2280873, MR1745012, MR3373236, MR0929094, MR3223291, MR0266853, MR4344032}. 
Motivated by our previous connections between compositions and Dyck paths, we show now arithmetic properties of certain classes of Dyck paths without explicit and ``simple'' closed forms in terms of formulas or generating functions. 

One such class are Dyck paths in which we restrict the number of allowed peaks per level. 
Recall from Section~\ref{sec:PeaksAndPaths} that a peak $\uu\dd$ has height $h$ if its $\uu$ step ends on height $h$.
\begin{definition}[Peak profile]
    \label{def:peakprofile}
    The \emph{peak profile} $(p_{i})_{i >0}$ of a Dyck path $\Pc$ is a sequence of non-negative integers such that $p_i$ is equal to the number of peaks of height $i$ in $\Pc$
\end{definition}

Previously, the equivalent concept of a \emph{peak heights multiset} appeared in the literature. 
For a given Dyck path, it is the multiset of all heights of peaks. 
Callan and Deutsch proposed the problem~\cite{CallanDeutsch2012PeakHeightsProblem} of counting all distinct peak height multisets (or equivalently, peak profiles) for Dyck paths of length $2n$. 
The answer is $2^{n}  - \sum_{k=0}^{n-1} p(k)$ where $p(k)$ is the number of integer partitions of $k$; see~\cite{Delorme2014PeakHeightsSolution} and \oeis{A208738}.
This formula can also be interpreted as the difference of the number of compositions of $n+1$ and the number of $1$s in all partitions of $n$~\cite[Result~4]{DastidarGupta2013Stanley}.

Note that in the proof of Theorem~\ref{theo:Dyckbound2}, we already encountered a special case of this concept in terms of the $2$-peak profile.
This definition is motivated by the profile of rooted trees, which associates to each tree the sequence $(t_i)_{i \geq 0}$, where $t_i$ is the number of nodes at depth $i$; see~\cite{Drmota2009}.
In particular, by the glove bijection~\cite{FlaSe2009}, which maps Dyck paths of length~$2n$ to rooted plane trees with $n+1$ nodes, the peaks of height $i$ are exactly the leaves at depth~$i$. 
Thus, the peak profile corresponds to the \emph{leaf profile} of rooted plane trees.

\smallskip

In this section, we will study Dyck paths with given peak profile.
For a given parameter $r \in \mathbb{N}$ we will focus on the class of profiles $P_r = \{ (p_i)_{i > 0} ~|~ \exists k \geq 0:  p_1 = \dots =p_k = r, p_i = 0 \text{ for } i > k \}$. In other words, the Dyck paths have exactly $r$ peaks at each reached depth. 
By the previous discussion, these are in bijection with trees that have exactly $r$ leaves in each reached level.
Let $D_r(n)$ be the number of Dyck paths whose profile is in $P_r$.
Note that the counting sequences for $r=1$ and $r=2$ are listed in the OEIS:
 $(D_1(n))_{n \geq 0} = (1, 1, 0, 2, 0, 4, 6, 8, 24, \dots)$ is \oeis{A287846} 
and 
$(D_2(n))_{n \geq 0} = (1, 0, 1, 0, 0, 3, 6, 0, 9, \dots)$ is \oeis{A287845}.
Even though these numbers are difficult to compute, our techniques will allow us to derive certain arithmetic properties. 
The main result of this section will be Theorem~\ref{theo:rplus1div}, in which we will show that $r+1$ divides $D_r(n)$ for $n>r$.

\smallskip 

As a motivation of such a result, let us first consider some simpler peak profiles.
First, let $\Delta = \{(p_i)_{i>0} ~:~ 0 \leq p_i \leq 1 \text{ for all } i >0 \}$, which means that no two peaks occur at the same height.
Let $D_{\Delta}(n)$ be the  number of such Dyck paths of length $n$. 
For these no closed form is known, but we can derive combinatorially the following congruence.

\begin{lemma}
    We have $D_{\Delta}(n) \equiv 1 \mod (2)$.
\end{lemma}

\begin{proof}
    For a given path, we can associate another path to it by traversing it from right-to-left instead of left-to-right. 
    This involution maps the path with only one peak to itself, 
    while it maps all other paths to a different path. 
    Hence the total number is always odd. 
\end{proof}

\smallskip

As a second introductory example, consider the profile class $S_r = \{(p_i)_{i>0} ~:~ \sum_{i >0 } p_i = r \}$ that characterizes Dyck paths with exactly $r$ peaks. 
It is well known, that the number $D_{S_r}(n)$ of associated Dyck paths of length $n$ is equal to the famous Narayana numbers $N(n,r)=\frac{1}{n} \binom{n}{r}\binom{n}{r-1}$.
The following divisibility result follows directly from \cite[Theorem~1.2]{BonaSagan2005NarayanaDiv}.

\begin{lemma}
    Let $e \geq 2$ and $n,k\geq0$ be integers. Then it holds that
    \begin{align*}
        N(2^e n + 2, 2^e k + 1) &\equiv N(2^e n + 2, 2^e k +2 ) \mod (2), \\
        N(2^e n + 2, 2^e k + i) &\equiv 0 \mod (2) & \text{ for } 3 \leq i \leq 2^e.
    \end{align*}
\end{lemma}

It remains to find a combinatorial interpretation of these results. 
As a direct corollary, one can consider the profile class $S_{\leq r} = \{(p_i)_{i>0} ~:~ \sum_{i >0 } p_i \leq r \}$ that characterizes Dyck paths with at most $r$ peaks. 
Then it holds that the associated number $D_{S_{\leq r}}(n)$ of Dyck paths of length $n$ is divisible by $2$ for all $n \equiv 2 \mod (2^e)$ and $r \not\equiv 1 \mod (2^e)$.

\smallskip

Our goal in this section is to study the divisibility properties of Dyck paths with exactly $r$ peask \emph{per reached level} associated with $P_r$.
For this purpose, we will need the class of \emph{little Schröder paths} that we discuss in the next section.

\subsection{Little and large Schr\"oder paths}

These are lattice paths with steps $\uu = (1,1)$, $\dd = (1,-1)$, and $\hh_2 = (2,0)$, without $\hh_2$ steps at height zero, that start at $(0,0)$, end at $(2n,0)$, and always stay weakly above the $x$-axis.
\emph{Large Schröder paths} are defined as above, but the steps $\hh_2$ are allowed everywhere.
Note that the semi-length is given by the sum of $\uu$ and $\hh_2$ steps.
The number of little Schröder paths of semi-length $n$ is given by~\oeis{A001003}.
When fixing the number of $\uu$ steps, they satisfy the following simple formula.
    
\begin{lemma}
    \label{lem:Schroederusteps}
    Let $s_{n,i}$/$\ell_{n,i}$ be the number of little/large Schröder paths of semi-length $n$ with $i$ steps $\uu$. They satisfy the following closed forms
    \begin{align*}
        s_{n,i} &= \frac{1}{n+1} \binom{n+i}{i} \binom{n-1}{i-1}, \\
        \ell_{n,i} &= \frac{1}{i+1} \binom{n+i}{i} \binom{n}{i}.
    \end{align*}
    The sequence for $s_{n,i}$ is \oeis{A033282}, also counting the dissections of a complex $(n+1)$-gon into $i$ regions.
    The sequence for $\ell_{n,i}$ is \oeis{A088617}.
\end{lemma}

\begin{proof}
    Define the bivariate generating functions $S(t,u) = \sum_{n,i \geq 0 } s_{n,i} t^n u^i$  and $L(t,u) = \sum_{n,i \geq 0} \ell_{n,i} t^n u^i$.
    From the definitions, we directly get
    \begin{align*}
        S(t,u) &= \frac{1}{1-tu L(t,u)}, \\
        L(t,u) &= 1 + tL(t,u) + tuL(t,u)^2.
    \end{align*}
    Therefore $S(t,u)$ satisfies the equation
    \begin{align*}
        t(u+1)S(t,u)^2 - (t+1)S(t,u) + 1 = 0.
    \end{align*}		
    Thus $D(t,u)=t S(t,u)$ is the generating function of dissections of a convex $(n+1)$-gon into $i$ regions; see~\cite[Theorem~3(i) and Section~3.1]{FlajoletNoy1999Non-crossing}.
    The coefficients of $D(t,u)$ can be extracted using Lagrange inversion. The result for large Schröder paths follows analogously.
\end{proof}

\begin{remark}
    An alternative derivation of $S(t,u)$ uses the well-known generating function $D(t,u) = \sum_{n,i\geq0} d_{n,i} t^n u^i$ of Narayana numbers 
    $d_{n,i} = \frac{1}{n} \binom{n}{i-1} \binom{n}{i}$: number of Dyck paths of semi-length $n$ with $i$ peaks (and $i-1$ valleys).
    Choose a subset of the $i-1$ valleys and replace each $\dd \uu$ in the chosen valley by $\hh_2$. 
    The semi-length stays the same and there are no $\hh_2$ steps on the $x$-axis. 
    Moreover, all such paths are different and we create all little Schröder paths.
    Hence, this proves $S(t,u) = D(t+1,u)$; see the comment by Paul Boddington in \oeis{A033282}.
\end{remark}

Before we give the main result of this section, let us comment on the following ``combinatorial curiosity''~\cite[p.~216]{FlajoletNoy1999Non-crossing} nicely linking with the results of the previous section. 
Let $s_r = \sum_{i \geq 0} s_{r,i}$ be the total number of little Schröder paths, also known as \emph{Schröder numbers}, \emph{super-Catalan numbers}, or \emph{Schröder--Hipparchus numbers}; see~\cite{Stanley1997Hipparchus}.

\begin{proposition}
    \label{prop:rplus1tuples}
    Dyck paths of height $2$ with exactly $r$ peaks at each height are in bijection with $(r+1)$-tuples of integer compositions (including the composition of $0$ into $0$) with grand total sum $r$. Both are enumerated by $(r+1)s_r$.
\end{proposition}

\begin{proof}
    Let Dyck path with the claimed properties be given.
    First, we decompose the path into an $r+1$ tuple by removing the $r$ hills (i.e., peak at height $1$).
    The first part is the subpath before the first hill, 
    the second part is the subpath between the first and second hill, and so on. 
    Note that these parts are either empty or contain only peaks at height $2$.
    
    Second, we map each part to an integer composition as in Theorem~\ref{theo:Dyckbound2}.
    If the part is empty, it is mapped to~$0$.
    Otherwise, we cut at the returns to zero and map each subpart to the number of peaks at height $2$.
    As the total sum of peaks at height $2$ is $r$ the grand total sum of the compositions is also $r$.

    For the inverse map, it suffices to reverse these steps.
    Finally, the enumeration sequence was derived in~\cite[p.~216]{FlajoletNoy1999Non-crossing}. Alternatively, it follows from the proof of Theorem~\ref{theo:rplus1div}; see Remark~\ref{rem:proofcntrtuples}.
\end{proof}


This result allows us to give the following interpretation for the central binomial coefficients. Note that it also directly gives a new interpretation for the Catalan numbers~$\frac{1}{r+1}\binom{2r}{r}$, imposing that number of 2-pyramids never exceeds the number of hills (or vice versa).

\begin{corollary}
    Dyck paths with exactly $r$ hills ($\uu \dd$) and $r$ pyramids of height $2$ ($\uu \uu \dd \dd$) are in bijection with Dyck bridges of semi-length $r$ and $(r+1)$-tuples of integer compositions consisting only of $1$s with grand total sum $r$.
    Hence, all these objects are enumerated by $\binom{2r}{r}$.    
\end{corollary}

\begin{proof}
    For the first bijection, replace each hill by an $\uu$ step and each pyramid of height $2$ by a $\dd$ step (or vice versa).
    The second part directly follows from Proposition~\ref{prop:rplus1tuples}.
    Combining both bijections, observe that the $i$th composition, corresponds to the run of up steps after the $i$th down step; see~\cite{AsinowskiHacklSelkirk2022Down}.
\end{proof}

\subsection{Dyck paths with exactly r peaks per reached level}

After this discussion of Schröder paths, we are ready to prove the following divisibility property.

\begin{theorem}
\label{theo:rplus1div}
Let $D_r(n)$ be the number of Dyck paths with semi-length $n$ and with exactly $r$ peaks for every reached height. 
Then $D_r(n) \equiv 0 \mod (r+1)$ for $n>r$.
\end{theorem}

\begin{proof}    
    \newcommand{\PP}{A} 
    \newcommand{\MS}{S} 
    \newcommand{\LOp}{\operatorname{\mathcal{L}}} 
    We will construct the paths recursively height-by-height, by lifting all possible paths of a given height by one.
    Let a Dyck path with $k$ returns to zero and exactly $r$ peaks for every reached height be given.
    Thus, it decomposes into a sequence of $k$ non-empty arches:
    \begin{align*}
        \PP_1 \PP_2 \dots \PP_k.
    \end{align*}
    Let us call the returns inside the path \emph{contacts}, i.e., the path above has $k-1$ contacts.
    Now we will cut at a subset $\MS$ of the contacts and lift the full paths by one:
    We attach an $\uu$ step at the beginning, $\dd$ at the end, and insert at each chosen contact a valley $\dd \uu$. 
    The height of this new path increased by one and its length by $i+1$.
    Furthermore, it has $i=|\MS|$ contacts and $i+1$ returns to zero. 
    This gives the new decomposition
    \begin{align*}
        \widetilde{\PP}_1 \widetilde{\PP}_2 \dots \widetilde{\PP}_{i+1}.
    \end{align*}
    Note that this new path still satisfies the constraint of having exactly $r$ peaks for each reached height greater than one.
    It remains to insert $r$ hills of the form $\uu\dd$ at ground level, which is equivalent to inserting at the beginning, the end, or between the new arches $\widetilde{\PP}_j$.
    For this purpose we choose a multiset consisting of $r$ of these $i+2$ positions. 
    In total this gives 
    $\binom{r+i+1}{r}$
    possibilities that fulfill the peak constraint.

    Note that the construction above creates all Dyck paths with the given constraint. 
    We can reverse it by removing all steps on ground level and concatenating the parts. The same reasoning as above gives the number of equivalent paths we create in each step.

    Let us now define the generating function $F(z,u) = \sum_{n,k} f_{n,k} z^n u^k$, where $f_{n,k}$ gives the number of Dyck paths with exactly $r$ peaks for each reached height of semi-length $n$ with $k$ returns to zero.
    Next we define a linear operator $\LOp_r$ that implements the construction discussed above:
    \begin{align}
        \label{eq:liftopexactlyr}
        \LOp_r(u^k) 
            = \sum_{i=0}^{k-1} \binom{r+i+1}{r} \binom{k-1}{i} (zu)^{r+i+1}.
    \end{align}
    The sum goes over all subset sizes of a $k-1$ element set, the term $\binom{k-1}{i}$ represents the number of such subsets, and the $+1$ in the power of $zu$ comes from the initial $\uu$ and final $\dd$ after a lift.
    Then we extend this operator linearly to polynomials in $u$, by treating the variable $z$ as a constant.
    Note that this operator can also be expressed in terms of hypergeometric functions
    \begin{align*}
        \LOp_r(u^k) 
               = (r+1) (zu)^{r+1} 
            \setlength\arraycolsep{1pt}
            {}_2 F_1\left(\begin{matrix}r+2,~1-k\\-2
            \end{matrix};-zu\right)
    \end{align*}
    This hypergeometric function has in general no integer coefficients, and in particular most terms are divided by $r+1$.
    Because of that this representation does not prove the divisibility property; however, it gives a strong hint.

    By the reasoning above, this allows us to define the functional equation
    \begin{align*}
        F(z,u) = (zu)^r + \LOp_r\left( F(z,u) \right).
    \end{align*}
    The term $(zu)^r$ corresponds to the unique path of height $1$ that consists of $r$ peaks: $(\uu \dd)^r$.
    Note that this functional equation has a unique formal power series solution, as the length in $z$ increases by at least one in each iteration.

    In order to prove the divisibility of $D_r(n)$ by $r+1$ for $n>r$, we will use the linearity of the operator: 
    We will show that all terms after one lifting operation fulfil this property, and therefore all terms in the sequence.
    We start with the only path of length $r$ given by $(\uu \dd)^r$ corresponding to the monomial $(zu)^r$.
    This path contains $r$ returns to zero, and we compute
     \begin{align}
        \label{eq:liftupdown}
        \LOp_r((zu)^r) 
            = z^r\sum_{i=1}^{r} \binom{r+i}{i} \binom{r-1}{i-1} (zu)^{r+i}.
    \end{align}
    Here, we rediscover the sequence $s_{r,i}$ from Lemma~\ref{lem:Schroederusteps}, yet multiplied by $r+1$.
    Hence, the claimed divisibility property follows.
\end{proof}

\begin{remark}
    \label{rem:proofcntrtuples}
    Note that setting $z=u=1$ in~\eqref{eq:liftupdown}, this proves the formula $(r+1)s_r$ for Dyck paths of height $2$ with $r$ peaks at each height from Proposition~\ref{prop:rplus1tuples}.
    Yet, setting only $z=1$ this gives a refined count, showing that $s_{r,i}$ counts such paths with exactly $i$ double-falls $\dd \dd$.

    However, it remains a combinatorial mystery why the Schröder numbers $s_r$ (or, rather, $(r+1)s_r$) appear here.
    In addition to the above interpretations for $s_{r,i}$, these numbers also enumerate
    plane trees whose with $r$ leaves and $i$ internal nodes of (out-)degree greater than or equal to $2$,
    as well as to
    plane trees of height at most $2$, with $r$ leaves each in levels $1$ and $2$ and $i$ internal nodes in level $1$.
    The latter interpretations follows from an application of the glove bijection (traversing the contour of a plane tree) to the Dyck paths of Proposition~\ref{prop:rplus1tuples}.
\end{remark}

The lifting technique of Theorem~\ref{theo:rplus1div} can also be used to study Dyck paths with different peak profiles. 
For example, the profile $P_{\leq r} = \{ (p_i)_{i > 0} ~|~ 0 \leq p_i \leq r \}$ corresponds to Dyck paths with at most $r$ peaks per level.
Let $E_r(n)$ be the number of such paths of length $n$. 
In order to enumerate them, one simply replaces in~\eqref{eq:liftopexactlyr} each $r$ in the summands by $j$ and adds an inner sum for $j$ from $0$ to $r$. 
For small values of $r$ these sequences are in the OEIS:
$(E_1(n))_{n \geq 0} = (1, 1, 1, 3, 5, 13, 31, 71, 
\dots)$ is \oeis{A281874},
$(E_2(n))_{n \geq 0} = (1, 1, 2, 4, 12, 31, 90, 
\dots)$ is \oeis{A287966},
$(E_3(n))_{n \geq 0} = (1, 1, 2, 5, 13, 40, 119, 
\dots)$ is \oeis{A287967}.
This information led us to the following result.

\begin{proposition}
    $E_1(n) \equiv 1 \mod (2)$.
\end{proposition}

\begin{proof}
    As discussed in the beginning of this section, by the glove bijection, Dyck paths with at most $1$ peak per level are in bijection with plane trees with at most $1$ leaf per level. 
    We now partition the trees of size $n$ into one singleton set and sets of size $k!$ for $k\geq 2$. 
    Thus, the total number of trees of size $n$ is odd. 

    First, the singleton set consists of the unary \emph{chain}, i.e., a tree that never branches. 
    Second, let a tree different from the chain be given. 
    When traversing this tree starting from the root, there is a first node that branches. 
    Let $k$ be its outdegree. 
    Then, we define a set associated to this tree by considering all $k!$ permutations of the respective children. 
    Note that all of these are distinct as each level consists of at most $1$ leaf.
\end{proof}

Up till $r=10$ we did not discover any further arithmetic properties for this class of peak profiles.
We invite the reader to study his/her favorite peak profile, and derive more such results.

\subsection{Congruences for the number of specific steps in paths}

As the number of up steps in Schröder paths was the key parameter for proving the previous congruences, we will discuss general congruences for the number of specific steps in general directed lattice paths at the end of this section. 
The paths start at $(0,0)$ and draw their steps from a finite step set $\Sc \subset \mathbb{Z}$.
For example, $\Sc=\{-1,1\}$ corresponds to steps of Dyck paths, and $\Sc=\{-1,0,1\}$ to steps of Motzkin paths.
As before, we call bridges paths ending on the $x$-axis, 
meanders paths never crossing the $x$-axis,
and excursions paths that are at the same time bridges and meanders.

\begin{lemma}
    \label{lem:bridgediv}
    Let $b_{n}(s)$ be the number of steps $s \in \Sc$ in all bridges of length $n$. Then, we have for $n>0$
    \begin{align*}
       b_{n}(s) &\equiv 0 \pmod{n}.
    \end{align*}
    Moreover, for a symmetric step set $\Sc = -\Sc$, we have for $n>0$
    \begin{align*}
       b_{n}(s) &\equiv 0 \pmod{2}.
    \end{align*}
    Hence, for a symmetric step set and odd $n$ it holds that $b_{n}(s) \equiv 0 \pmod{2n}$.
\end{lemma}

\begin{proof}
    The results follow from transformations in the spirit of the previous sections.
    For the first result, we build cyclic shift inspired by the cycle lemma.
    Let $\Bc$ be a bridge of length $n$. 
    Then, there exists a unique representation of $\Bc$ as the concatenation of a minimal bridge $\Bc'$ of size $n'=\frac{n}{m}$ repeated $m$ times, i.e.,
    \[
        \Bc = (\Bc')^m.
    \]
    Note that $m \geq 1$, as $\Bc'=\Bc$ is a candidate, yet in general not the minimal one.
    Let $b(s)$ and $b'(s)$ be the number of steps $s$ in $\Bc$ and $\Bc'$, respectively. 
    Then, the above decomposition shows that $b(s) = m b'(s)$.
    Let $\sigma(\Bc)$ be the cyclic shift of the steps in $\Bc$, i.e., the first step of $\Bc$ is moved to the end.
    Observe that we have 
    \begin{align*}
        \sigma(\Bc) = \sigma(\Bc')^m.
    \end{align*}
    Moreover, due to the minimality of $\Bc'$, the first $n'-1$ cyclic shifts lead to mutually distinct bridges $\sigma^k(\Bc)$ (equivalently, different $\sigma^{k}(\Bc')$),
    while $\sigma^{n'}(\Bc) = \Bc$.
    Hence, the orbit of this operation consists of $n'$ different bridges. 
    Since each of these bridges has $m b'(s)$ steps $s$, the total number of steps $s$ in this orbit is equal to $n' m b'(s) = n b'(s)$.
    Finally, note that these orbits form a partition of all bridges and therefore the result follows.

    For the second result, we apply the repeatedly used involution on bridges of mirroring along the $x$-axis.
    Thereby, each bridge is mapped into a different bridge in which the positions of $s$ and $-s$ swap. 
    The final claim follows directly by combining the previous two results.
\end{proof}

\begin{example}
    For Dyck bridges with step set $\Sc = \{-1, 1\}$ it is easy to see that $b_{2n}(s) = n \binom{2n}{n}$ (and $b_{2n+1}(s)=0$), as every Dyck bridges is composed of an equal number of steps $-1$ and $1$.
    Hence, our result implies that $\binom{2n}{n} \equiv 0 \pmod{2}$
    and $\binom{4n+2}{2n+1} \equiv 0 \pmod{4}$ for $n >0$.
    Alternatively, this result follows from Kummer's famous result~\cite{Kummer1852Carries} that the highest power of a prime $p$ dividing $\binom{m+n}{n}$ is equal to the number of carries in the addition of $m$ and $n$ in base~$p$; for more details see~\cite{SpiegelhoferWallner2018Divisibility}.
\end{example}

\section{Restrictions on peaks and parts}

\subsection{Restricted summits}
\label{sec:restrictedsummits}

We have previously seen the notion of peaks $\uu \dd$ in Dyck paths.
Now, we want to analyze them in more detail in Dyck bridges. 
As bridges may go below the $x$-axis, we introduce the following more general notion of summits.

\begin{definition}[Summit]
A summit in a path is a point that has a larger ordinate than all its direct neighbors.
\end{definition}

Therefore, in a Dyck bridge (and walk) three types of summits may occur: 
A point inside the path if it is part of the pattern $\uu \dd$, 
the starting point is a summit if the first step is $\dd$,
the end point is a summit if the last step is $\uu$.
Let us consider Dyck bridges with non-decreasing summits.
Dyck paths with the same property where studied in~\cite{PenaudRoques2002Peaks}; see~\oeis{A048285}.

\begin{theorem}
    The number of Dyck bridges of semi-length $n$ with non-decreasing summits is equal to the odd Fibonacci number $F_{2n+1}$ (i.e., $F_0=0$, $F_1=1$, $F_{n} = F_{n-1} + F_{n-2}$ for $n\geq 2$) given by~\oeis{A001519}.    
    
\end{theorem}

\newcommand{\SEQ}{\operatorname{SEQ}}
\newcommand{\Pcf}{\overline{\Pc}}
\begin{proof}[Proof 1 (Generating functions)]
    We give a first proof using generating functions.
    The key observation of non-decreasing summits is the following: every non-terminal sequence of down steps has to be followed by a sequence of up steps of at least the same length.

    We perform a first-passage decomposition and cut the bridge at the first up steps leading to a new level.
    If there are any summits at this level they have to appear before the next up step leading to the next level. 
    Moreover, no summits may be below this level.
    Thus, these ``first'' up steps are followed by (possibly empty) sequences of flipped pyramids $\dd^k \uu^k$. 
    Let $\Pcf = \{\dd^k \uu^k : k\geq 1\}$ be the set of  flipped pyramids.
    Therefore, we get the following decomposition 
    \[
        \SEQ\left( \SEQ(\Pcf) \uu \right) \SEQ(\Pcf) \dd^h,
    \]
    where $h$ is the height of the last up step, thus, returning to the $x$-axis.
    This translates into the following generating function.
    Note that as we count paths by semi-length, it suffices to weight up steps by $z$.
    \[
        \frac{1}{1 - \frac{z}{1 - \frac{z}{1-z}}}\frac{1}{1 - \frac{z}{1-z}}
        = \frac{1-z}{1-3z+z^2}.
    \]
    This is the generating function of odd Fibonacci numbers, which proves the claim.
\end{proof}

\begin{proof}[Proof 2 (Bijection).]
    Deutsch and Prodinger~\cite{DeutschProdinger2003Bijection} show that several combinatorial objects are counted by the odd Fibonacci numbers.
    A particular class is given by Dyck paths of semi-length $n+1$ with non-decreasing valley heights. 
    We will now present an explicit bijection between these and Dyck bridges of semi-length $n$ and non-decreasing summits.
    
    The key observation is that both classes satisfy related decompositions. 
    For this purpose recall that $\Pcf = \{\dd^k \uu^k : k\geq 1\}$ is the set of flipped pyramids, and let $\Pc = \{\uu^k \dd^k : k\geq 1\}$ be the set of pyramids.
    Let a Dyck bridge be given, and decompose it using a first-passage decomposition into
    \begin{align*}
        B_0 \uu B_1 \uu \dots B_k \dd^k,
    \end{align*}
    where $B_i = \SEQ(\Pcf)$ for $i=0,\dots,k$.
    Next, we flip each part $B_i$, by exchanging $\uu$ and $\dd$ steps mapping flipped pyramids to pyramids. Moreover, we insert a peak $\uu \dd$ before the final $\dd^k$, in order to remember where to cut.

    For the reverse bijection, we decompose the Dyck path with non-decreasing valleys using a last-passage decomposition, i.e., we cut at the last time we leave a certain height. This gives
    \begin{align*}
        D_0 \uu D_1 \uu \dots D_k \, \uu \dd \, \dd^k,
    \end{align*}
    where $D_i = \SEQ(\Pc)$ for $i=0,\dots,k$.
    Observe that these paths necessarily end with an $\uu \dd$ before the final run of $\dd$ steps.
    Now, we remove this $\uu \dd$ and flip each part $D_i$ to regain bridges.
\end{proof}

\begin{remark}[More bijections with summits]
    Note that Deutsch and Prodinger~\cite{DeutschProdinger2003Bijection} discuss many further bijections and enumeration results, which can now be combined with our results and promising many further interesting bijections worth investigating. 
    For example, the number of summits minus one is equal to the number of columns in column convex polyominoes of area $n+1$.
\end{remark}

\begin{remark}[OEIS connections]
    When we do not allow the full path to be below the $x$-axis, we get \oeis{A061667}: $F_{2n+1}-2^{n-1}$. This can be seen from the generating function, by subtracting the generating function for sequences of pyramids given by $\frac{1-z}{1-2z}$.
    
    When additionally the end point is not considered as a summit, i.e., bridges could end with $\uu$, then the sequence is \oeis{A105693}: $F_{2n+2}-2^{n}$.
    To prove this result, simply multiply the previous generating function by the one for one (possibly empty) pyramid given by $1/(1-z)$.

    Note that the class of paths when additionally the paths weakly below the $x$-axis are allowed, is not in the OEIS. The counting sequence starts as $(1,2,6,17,47,128,345,\dots)$.

    In addition, let us return to peaks, i.e., $\uu \dd$ inside the path.
    When these peaks may start below the $x$-axis, but still have to be non-decreasing, we get \oeis{A094864}.
    This sequence also counts the number of one-cell columns in all directed column-convex polyominoes of area $n$; see~\cite[Lemma~3.4]{BarucciEtal1993DCC}.
    Moreover, this sequence also connected to special cells in stacks of area~$n+1$; 
    see~\cite[Equation~7]{RinaldiRogers2006Stack}.
\end{remark}

\begin{remark}[Asymptotic consequences]
    Asymptotically, the numbers of Dyck excursions and bridges of semi-length $n$ with non-decreasing summits only differ in the multiplicative constants; both satisfy $cst \cdot (\frac{3 + \sqrt{5}}{2})^n$; see~\cite[Theorem~2]{PenaudRoques2002Peaks}. 
    For excursions the constant is $c_{E} \approx 0.11998$ whereas for bridges it is $c_{B} \approx 0.7236$.
    Therefore, asymptotically, about one out of $6$ bridges with non-decreasing \emph{summits} is an excursion.
\end{remark}

\medskip

We conclude this section with another class of compositions linked to Fibonacci numbers.

\begin{definition}[Decorated Compositions]
A composition of $n$ is called a \emph{decorated composition} if any $k$ occurring in the composition can come with $k$ decorations/labels. 
\end{definition}

For example, the $8$ decorated compositions of $3$ are $3_1$, $3_2$, $3_3$, $2_1+1_1$, $2_2 + 1_1$, $1_1+2_1$, $1_1+2_2$, and $1_1+1_1+1_1$.

\begin{proposition}
    The number of decorated compositions of $n$ is equal to the even Fibonacci number $F_{2n}$; see~\oeis{A001906}.     
\end{proposition}

\begin{proof}
    There are $k$ decorations for a part $k$ therefore the number of compositions of $n$ into $m$ parts is given by 
    $(x +2x^2 +3x^3+ \dots)^m$. 
Summing over $m$, gives
$
\sum_{m=1}^{\infty}\frac{x^m}{(1-x)^{2m}} 
= \frac{x}{1-3x+x^2}.
$
This is the required generating function for the even Fibonacci numbers. 
\end{proof}

\begin{remark}
    The same proof also shows that the even Fibonacci numbers also enumerate the norms of the compositions of $n$. 
    The \emph{norm of a composition} is the product of its parts, as defined for partitions, see, e.g.,~\cite{SchneiderSills2020Norm}.
\end{remark}

\subsection{First, Greatest and Smallest Parts}
\label{sec:firstpart}

In the theory of partitions it is a well established phenomenon that the first part and the last part govern the shape of the partition \cite{zbMATH06204634, zbMATH07655649,zbMATH07199594}. Nathan Fine in his manuscript \lq \lq Some New Results on Partitions" had left several intriguing theorems on partitions in the Proceedings of the National Academy of the Sciences but the proofs were not published subsequently. George Andrews found the subsequent proofs in Fine's  manuscript later and went on to generalise these theorems connecting them to the Rogers--Ramanujan identities. 

The theorem that we are going to discuss below has a striking connection to Euler's famous Pentagonal Number Theorem. Nathan Fine himself writes \cite{MR0027798} that the theorem \lq \lq bears some resemblance to the famous pentagonal theorem of Euler, but we have not been able to establish any real connection between the two theorems."

\begin{theorem}
Let  $D_e(n)$  and $D_o(n)$ be the sets of partitions \( \lambda \) of \( n \) into distinct parts, such that the first part $\lambda_1$  is even and odd, respectively. Then:
\[
D_e(n) - D_o(n) =
\begin{cases}
1, & \text{if } n = \frac{k(3k + 1)}{2}, \\
-1, & \text{if } n = \frac{k(3k - 1)}{2}, \\
0, & \text{otherwise}.
\end{cases}
\]
\end{theorem}

We obviously notice from the start that the sign changes occur at pentagonal numbers. It is only recently that Igor Pak proved the above theorem using an argument similar to Sylvester and Franklin's famous constructive proof of the pentagonal number theorem thereby proving a connection of this theorem to Euler's celebrated result in his paper~
\cite{MR1962923}. 
In a recent article~\cite{Dastidar2024Parity2} we examine the following flipped question for integer partitions. 
\begin{theorem}
Let  $V_e(n)$  and $V_o(n)$ be the sets of partitions \( \lambda \) of \( n \) into distinct parts, such that the last part $\lambda_1$  is even and odd, respectively. Then for $n> 6$ we get:
$V_o(n) > V_e(n)$.
\end{theorem}

We can find a similar theorem for integer compositions into distinct parts with respective first parts being odd (resp.\ even). 
But first we begin with a simpler setting, i.e.\ integer compositions such that the first part is odd (resp.\ even). 
Let $C_o(n)$ denote the number of integer compositions such that the first part is odd and $C_e(n)$ denote the number of compositions where the first part is even. 
\begin{theorem}
\label{theo:ConminusCenandTriangle}
Let $C_o(n)$ denote the number of integer compositions such that the first part is odd and $C_e(n)$ denote the number of compositions where the first part is even.  Let $G(n)$ denote the number of closed walks on a triangle (starting from a chosen vertex). Then:
\begin{align}
    \label{eq:CominusCeisGn1}
  C_o(n)- C_e(n) &=   G(n-1).
\end{align}
\end{theorem}

\begin{proof}[Proof 1 (Generating function and transfer matrix)]
We start with the generating function of the number of compositions such that the first part is odd. 
%
The number of compositions of $n$ into $k$ parts such that the first part is $(2r-1)$ is given by the coefficient of $x^n$ in $x^{2r-1}(x+x^2+x^3+\dots)^{k-1}$. 
Therefore the number of compositions of $n$ into $k$ parts such that the first part is odd is given by $\frac{x}{1-x^2}\frac{x^{k-1}}{1-x^{k-1}}$.
So 
we sum over $k$ to find 
\begin{align*}
\sum_{n=0}^{\infty}C_o(n)x^n  & = \sum_{k=1}^{\infty}\frac{x}{1-x^2}{\frac{x^{k-1}}{1-x^{k-1}}} 
= \frac{x}{1-x-2x^2}.
\end{align*}

Similarly we can find the number of compositions of $n$ such that the first part is even to be the coefficient of $x^n$ in the expansion of $\frac{x^2}{1-x^2}\frac{1-x}{1-2x}= \frac{x^2}{1-x-2x^2}$. 
Therefore, we see that 
\begin{align*}
\sum_{n=0}^{\infty}(C_o(n)- C_e(n))x^n 
= \frac{x(1-x)}{1-x-2x^2}.
\end{align*}
Decomposing into partial fractions and simplifying we can see that the difference between the number of integer compositions with first part respectively odd and even is given by $C_o(n)- C_e(n) = (2^{n-1} + 2 \, (-1)^{n-1})/3$.

Now the adjacency matrix for the triangle is simply 
$A = \left(\begin{smallmatrix}
0 & 1 & 1 \\
1 & 0 & 1 \\
1 & 1 & 0
\end{smallmatrix}\right)$.
If we consider the $n$th power of $A$, the entry in the $i$th row and $j$th column  gives the number of walks of length $n$ from vertex $i$ to vertex $j$. 
Thus, the number of closed walks are the diagonal entries of $A^n$.
Due to symmetry and the regular degree of each vertex, all diagonal entries are equal.

The eigenvalues of matrix $A$ are $-1$ (with multiplicity $2$) and $2$ (with the multiplicity $1$). 
Therefore, the number of closed walks of length~$n$ starting from any of the three vertices and ending at the same vertex is $2^n + 2 \, (-1)^n $.
Dividing by $3$, we get the number of closed walks starting from a definite vertex and ending there which proves our theorem. 
\end{proof}

\begin{proof}[Proof 2 (Bijection)]
First we show that $C_o(n) = C_e(n) + 2C_e(n-1)$. 
Let a composition of $n$ with odd first part $n_1$ be given.
If $n_1>1$ then subtracting $1$ from the first part, gives a composition of $n-1$ with even first part enumerated by $C_e(n-1)$. 
If $n_1=1$ then removing $n_1$ gives an unconstrained composition of $n-1$ that is enumerated by $C_o(n-1)+C_e(n-1)$. 
Therefore, $C_o(n) = C_o(n-1)+2C_e(n-1)$. 
Finally, note that $C_o(n-1)=C_e(n)$ as adding $1$ to the first part is a bijection.
Thus, Claim~\eqref{eq:CominusCeisGn1} is equivalent to 
\begin{align*}
    2 C_e(n) = G(n),
\end{align*} 
which we will prove now bijectively.

In the triangle, we fix a root vertex and label it by $1$ and the other vertices by $2$ and $3$.
This allows us to interpret the factor $2$ as a fixed first step from vertex $1$ to $2$, and it remains to show that these closed walks are enumerated by $C_e(n)$.
For this purpose, we assign to each edge $(a,b)$, connecting vertices $a$ and $b$, a label $\ell(a,b)$ as follows: 
\begin{align*}
    \ell (1,2)  = 2, 
    \qquad
    \ell(1,3) = \bar{2}, 
    \qquad
    \ell(2,3) = \ell(3,2) = 1,
    \qquad
    \ell(2,1) = \ell(3,1) = \varepsilon,
\end{align*}
where $\varepsilon$ denotes an empty label. 
The key observation is that each closed path from~$1$ to~$1$ of length $n$ is in bijection with a word $\{1,2,\bar{2}\}^n$, by reading the edge labels along its traversal.
Observe that in each such traversal, each step $(1,2)$ (resp., $(1,3)$) needs to be followed by a $(2,1)$ or $(3,1)$ step, because the path is closed. 
Therefore, the sum of the labels in the word is equal to the length of the path.

Finally, each such word is in bijection with a composition of $n$ with even first part. 
The idea is similar to Proposition~\ref{prop:threecompbijection}: 
We go through the word from left to right.
Each time we read a $\bar{2}$ we remove it and increase the left neighbor by $2$. 
After all parts $\bar{2}$ have been removed, what remains is a composition with even first part, because the path started with a step from~$1$ to~$2$. 

For the reverse bijection, replace each even part $2m>2$ by the word $2\bar{2}\dots\bar{2}$ containing $m-1$ letters $\bar{2}$, and each odd part $2m+1>2$ by $1\bar{2}\dots\bar{2}$ consisting of $m$ letters $\bar{2}$. 
In other words, the sum of parts ($\bar{2}$ is treated like $2$) gives $2m$ (resp., $2m+1$).
By the above mapping, these words then correspond to a unique path in the triangle.
\end{proof}

\begin{remark}
The previous bijection is a combinatorial interpretation of the underlying recurrence relation (proved above)
\begin{align}
    \label{eq:CeFibonaccilike}
    C_e(n) = C_e(n-1) + 2 \, C_e(n-2),
\end{align}
for $n \geq 2$, where $C_e(1)=0$ and $C_e(2)=1$.
As $C_o(n) = C_e(n-1)$, the same recurrence with shifted initial conditions $C_o(1)=1$ and $C_o(2)=1$ holds for compositions starting with an odd part.

More generally, let $m>0$ be a positive integer. 
Then the same combinatorial construction works for compositions $(n_1,n_2,\dots,n_k)$ such that $n_1 \equiv i \mod (m)$ for a fixed $i \in \{0,1,\dots,m-1\}$. 
In particular, let $C_i(n)$ be the number of such compositions of $n$. 
Then, it is easy to show
\begin{align*}
    C_i(n) = C_i(n-1) + \dots + C_i(n-m+2) + 2 \, C_i(n-m+1),
\end{align*}
for $n \geq m$ and suitable initial conditions.
Then, these are in bijection with words in the alphabet $\{1,2,\dots,m-1,m,\bar{m}\}$ such that the sum of letters is equal to $n$ ($\bar{m}$ is treated like $m$).
Then, one can also associate to these compositions walks in a (multi-)graph with the labels from the alphabet.
However, it seems to be only ``nice'' for $m=2$, i.e., the triangular case.
\end{remark}

\begin{remark}
    A third proof idea of Theorem~\ref{theo:ConminusCenandTriangle} builds on $C_o(n) = C_e(n-1)$ combined with the observation that $C_o(n) + C_e(n) = 2^{n-1}$, as every composition has either an odd or even first part.
    Thus, one gets the recurrence $C_e(n)+C_e(n-1) = 2^{n-1}$, which can be solved using standard methods. 
    
    Note that this approach also works for $C_i(n)$ for any modulus $m$, as $C_i(n) = C_{i+1}(n+1)$ for $i \in \{1,2,\dots,m-1\}$, and obviously $C_0(n) + \dots + C_{m-1}(n) = 2^{n-1}$.
\end{remark}

\medskip

In a similar manner we will study the influence of the first steps on the behavior of lattice paths. 
We will mention one specific intriguing result from recent times due to Andrei Asinowski regarding Dyck paths where the first peak is the highest; see \oeis{A287709}.

\begin{theorem}[{\!\!\cite{Asinowski2024Private}}]
    \label{theo:Asinowski}
    The number of Dyck paths of semilength $(n-1)$ whose peak of maximum height is attained by the initial ascent is equal to the number of Dyck paths of semilength $n$ such that every peak at height $h > 1$ is preceded by at least one peak of height $h-1$.
\end{theorem}

The analogous question is easily solved for Dyck bridges.

\begin{lemma}
    \label{lem:DyckbridgeFirstPeak}
    The number of Dyck bridges  such that the first peak is the highest is enumerated by the Catalan numbers. 
\end{lemma}

\begin{proof}
We start by decomposing a Dyck bridge in the following fashion; see Figure~\ref{fig:DyckLastHeight}:
We mark the points where the path reaches a particular height for the last time. 
Then one down step from each of those points to the next marked point forms a (possibly empty) negative Dyck path, which  becomes a usual Dyck path after reflection. 
Because the first peak is the highest (of height $k$ say) the path may be decomposed as 
$B_k(z)$ = $z^{k+1}D(z)^kz^{k-1}$
= $z^{2k}D(z)^k$, where $D(z)$ is the generating function of Dyck paths/Catalan numbers. 
Therefore summing over $k$ we obtain all possible Dyck bridges where the first peak is the highest one:
\begin{align*}
\sum_{k=0}^{\infty} z^{2k}D(z)^k 
  = \frac{1}{1 - z^2D(z)}  = D(z),
\end{align*}
since $D(z) = 1 + z^2D(z)^2$.
\end{proof}

\begin{remark}
As demonstrated repeatedly in Section~\ref{sec:PeaksAndPaths}, we can rearrange and flip the decomposition in Lemma~\ref{lem:DyckbridgeFirstPeak} into a sequence of arches giving a classical Dyck path in a bijective fashion.     
However, Theorem~\ref{theo:Asinowski} has so far no bijective explanation and the problem seems rather difficult.
\end{remark}

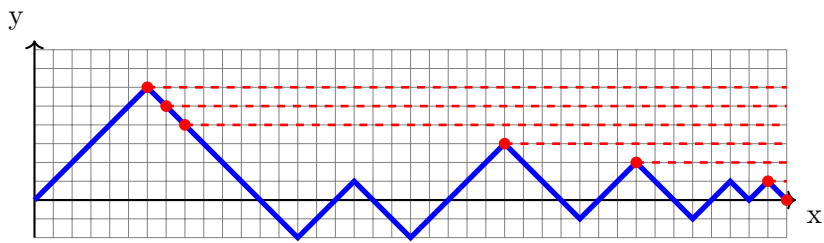
\begin{figure}[t]
    \centering
    \begin{tikzpicture}[scale=0.25]
        \draw[help lines] (0,-2) grid (40,8);
        \draw[thick,->] (0,0) -- (40.5,0) node[anchor=north west] {x};
        \draw[thick,->] (0,-2) -- (0,8.5) node[anchor=south east] {y};
    
        \draw[line width=2pt,blue] (0,0) -- (1,1) -- (2,2) -- (3,3) -- (4,4) -- (5,5) -- (6,6) -- (7,5) -- (8,4) -- (9,3)
            -- (10,2) -- (11,1) -- (12,0) -- (13,-1) -- (14,-2) -- (15,-1) -- (16,0) -- (17,1) -- (18,0) -- (19,-1)
            -- (20,-2) -- (21,-1) -- (22,0) -- (23,1) -- (24,2) -- (25,3) -- (26,2) -- (27,1) -- (28,0)
            -- (29,-1) -- (30,0) -- (31,1) -- (32,2) -- (33,1) -- (34,0) -- (35,-1) -- (36,0) -- (37,1)
            -- (38,0) -- (39,1) -- (40,0);

        \draw[line width=1pt, dashed, red] (6,6) -- (40,6); 
        \draw[line width=1pt, dashed, red] (7,5) -- (40,5); 
        \draw[line width=1pt, dashed, red] (8,4) -- (40,4); 
        \draw[line width=1pt, dashed, red] (25,3) -- (40,3); 
        \draw[line width=1pt, dashed, red] (32,2) -- (40,2); 
        \draw[line width=1pt, dashed, red] (39,1) -- (40,1); 
        
        \filldraw[red] (6,6) circle (8pt); 
        \filldraw[red] (7,5) circle (8pt); 
        \filldraw[red] (8,4) circle (8pt);
        \filldraw[red] (25,3) circle (8pt); 
        \filldraw[red] (32,2) circle (8pt);
        \filldraw[red] (39,1) circle (8pt);
        \filldraw[red] (40,0) circle (8pt);
    
    \end{tikzpicture}
    \caption{Last-passage decomposition of a Dyck bridge used in Lemmas~\ref{lem:DyckbridgeFirstPeak} and \ref{lem:Dyckbridgeirreduciblepairs}}
    \label{fig:DyckLastHeight}
\end{figure}

Modifying Asinowski's theorem for the case of Dyck bridges we get back a surprising connection to pairs of compositions again.
Bender, Lawler, Pemantle and Wilf~\cite{bender2003irreducible} define irreducible ordered pairs of compositions as follows: 
Let $n = b_1 + \dots + b_k = b'_1 + \dots + b'_k$ be a pair of compositions of $n$ into $k$ positive parts. This pair is called \emph{irreducible} if there is no positive $j < k$ for which $b_1 + \dots + b_j = b'_1 + \dots + b'_j$. 
Note that Callan~\cite{callan2012identity} showed that these are also in bijection with hill-free Dyck bridges.
We will now present a bijection with a different subfamily of Dyck bridges.

\begin{lemma}
    \label{lem:Dyckbridgeirreduciblepairs}
  The number of Dyck bridges with semi-length $n$, featuring a single highest peak and initiating with a hill, where each peak of height $h>1$ is preceded by a peak of height $h-1$, corresponds precisely to the number of irreducible ordered pairs of compositions of $n$.
\end{lemma}

\begin{proof}
 We split the path into two parts such that the first part starts from the origin and continues up to the vertex of the highest peak. This is the point where the second part starts and ends when the path finally ends on the $x$-axis. 

 The first part is decomposed such that we mark the last time the path crosses each successive height. So in effect we have a chain beginning with an up step followed by a non-empty Dyck path followed by an up step and a non-empty Dyck path and so on, up to the highest peak. Assuming that the highest peak has height $k$, the path decomposes as $z^k(D(z)-1)^{k-1}$ (since the highest peak is just $\uu\dd$).

 The second part we decompose as we have done just in the previous lemma, such that we mark the point when path reaches a particular height for the last time. Then one down step from each of those points to the next marked point forms a (possibly empty) negative Dyck path and so the path decomposes into $z^kD(z)^k$. 

 So gluing them together gives us the entire path and to find all of these paths we sum over $k$ which gives us the generating function of all such paths as 
\begin{align*}
\sum_{k=1}^{\infty} z^{2k}(D(z)-1)^{k-1}D(z)^k 
  = \frac{z^2D(z)}{1 - z^2(D(z)-1)D(z)} 
  = \frac{z^2}{z^2 + \sqrt{1 - 4z^2}}.
\end{align*}
By~\cite[Theorem~1]{bender2003irreducible}, $\frac{z}{z + \sqrt{1 - 4z}}$ is the generating function of the number of irreducible ordered pairs of compositions of $n$, and the claim follows.
\end{proof}

\begin{remark}
We may also ask the opposite question: 
What are the number of Dyck bridges such that the first peak is of the smallest height?    
We can prove that such paths are enumerated by the partial sums of Catalan numbers. We leave it as an open problem to explore whether there is any connections to these numbers and integer compositions in some paradigm. 
\end{remark}

\section*{Acknowledgements}

This research was supported by the Austrian Science Fund (FWF): P~34142.


\bibliographystyle{mybiburl}
\bibliography{bibliography}    
  
\end{document}